\documentclass[11pt,reqno]{amsart}

\DeclareMathAlphabet\mathbfcal{OMS}{cmsy}{b}{n}
\usepackage{xcolor}
\usepackage{amssymb}
\usepackage{amscd}
\usepackage{amsfonts}
\usepackage{mathrsfs}
\usepackage{setspace}
\usepackage{version}
\usepackage{mathtools}
\usepackage{textcomp}
\usepackage{stmaryrd}
\usepackage[english]{babel}
\usepackage[colorlinks=true,linkcolor=blue]{hyperref}
\usepackage[nice]{nicefrac}

\newtheorem{theorem}{Theorem}[section]
\newtheorem{lemma}[theorem]{Lemma}
\newtheorem{proposition}[theorem]{Proposition}
\newtheorem{corollary}[theorem]{Corollary}
\theoremstyle{definition}
\newtheorem{definition}[theorem]{Definition}

\newtheorem{discussion}[theorem]{Discussion}
\newtheorem{example}[theorem]{Example}

\theoremstyle{remark}
\newtheorem{remark}[theorem]{Remark}
\numberwithin{equation}{section}

\newcommand{\esssup}{{\rm ess.sup}}

\newcommand{\field}[1]{\mathbb{#1}}

\newcommand{\R}{\field{R}}
\newcommand{\N}{\field{N}}

\begin{document}

	\title[Maxitive functions with respect to general orders]{Maxitive functions with respect to general orders}

    \author{Michael Kupper}
	\address{Department of Mathematics and Statistics, University of Konstanz}
	\email{kupper@uni-konstanz.de}
	
		\author{Jos\'e M.~Zapata}
	\address{Universidad de Murcia.  Dpto. de Estadística e Investigación Operativa,  30100 Espinardo, Murcia, Spain}
	\email{jmzg1@um.es}
	
\thanks{ We wish to express our gratitude to two anonymous referees for their careful reading of the manuscript and for their comments, which significantly improved the presentation. 
The second author was partially supported by Fundación Séneca - ACyT Región de Murcia project 21955/PI/22, Agencia Estatal de Investigación (Government of Spain) and Projects PID2021-122126NB-C32 and PID2022-137396NB-I00   funded by MICIU/AEI/10.13039/ 501100011033/ and FEDER A way of making Europe.}

		\date{\today}

	\subjclass[2010]{}

\begin{abstract} 
In decision-making, maxitive functions are used for worst-case and best-case evaluations. Maxitivity gives rise to a rich structure that is well-studied in the context of the pointwise order. In this article, we investigate maxitivity with respect to general preorders and provide a representation theorem for such functions. The results are illustrated for different stochastic orders in the literature, including the usual stochastic order, the increasing convex/concave order, and the dispersive order.  

\smallskip\noindent
 \emph{Key words:} Maxitive functions, representation theorem, stochastic orderings 
   
   \mbox{}\\
 		\smallskip
		\noindent \emph{AMS 2020 Subject Classification: 62C86, 47H05, 28A25, 91B05} 
	\end{abstract}

	\maketitle
	
	\setcounter{tocdepth}{1}

\section{Introduction}
In decision-making under risk and uncertainty, one typically considers a set of gambles \( X \) and a function \( \psi \colon X \to \mathbb{R} \), which assigns to each gamble \( f \in X \) a value \( \psi(f) \), representing its price, risk, or utility. In the context of imprecise probabilities, such functions are referred to as  lower/upper previsions; see \cite{walley}. In financial mathematics, they are commonly called risk measures \cite{follmer} or nonlinear expectations \cite{peng}.

\smallskip
In many cases, the function \( \psi \) takes the form of an expectation, where the value \( \psi(f) \) represents the weighted average of the values associated with the gamble \( f \). However, decisions are often based on assessments of best-case or worst-case scenarios, rather than relying on averages. These evaluations are captured by \emph{maxitive functions}. A function \( \psi \) is called maxitive if it satisfies the condition \( \psi(f \vee g) = \max\{\psi(f), \psi(g)\} \) for any two gambles \( f \) and \( g \), where \( f \vee g \) denotes the pointwise supremum of \( f \) and \( g \).

\smallskip
Maxitive functions play a crucial role in \emph{possibility theory}~\cite{dubois, dubois2, zadeh} and serve as aggregation functions in various fields such as economics, information fusion, and control theory~\cite{grabisch, torra}. Under reasonable assumptions, these functions can be represented in terms of maxitive integrals, such as the Shilkret integral discussed in~\cite{cattaneo,kupper2, poncet,shilkret}. Beyond decision theory, maxitive functions also arise in other areas, including idempotent or tropical analysis~\cite{maslov} and large deviations theory~\cite{kupper, kupper2, puhalskii}. 

\smallskip
In probability theory, statistics, and particularly decision theory, a wide range of criteria exist to define order relations between gambles beyond the pointwise order~\cite{muller, shaked2007}. While most of the literature has focused on maxitivity with respect to pointwise order, less attention has been paid to maxitivity with respect to more general orders. The aim of this work is to fill this gap by studying maxitivity in the context of preorders, which enables the extension of the maxitive structure and its beneficial properties to a wider class of functions. By passing from the pointwise order to weaker \emph{stochastic orders}, new maxitive functions are obtained. For example, in~\cite{wang}, it is shown that the value-at-risk is maxitive with respect to the usual stochastic order.

\smallskip  
The main result of this paper provides a characterization and \emph{representation of countably maxitive functions} on a preordered space.  
The proof relies on recent representation results from \cite{delbaen,max} for convex monotone functions on spaces of bounded continuous functions, and incorporates dual representations of sublevel sets as discussed in \cite{drapeau}. A key assumption is that the preorder \( \preccurlyeq \) admits a numerical representation given by \( f \preccurlyeq g \) if and only if \( \phi(f) \leq \phi(g) \) for all \( \phi \) in some Polish space\footnote{I.e., a separable completely metrizable topological space.} \( \Phi \), where the mapping \( \phi \mapsto \phi(f) \) is lower semicontinuous\footnote{I.e., $\{\phi\in\Phi\colon \phi(f)\le s\}$ is closed for all $s\in\mathbb{R}$.} for all \( f\in X\).
 This assumption is fulfilled for typical stochastic orders. As an application of the main result, we obtain characterizations of maxitive functions with respect to the usual stochastic order, the increasing convex order, the convex order, the increasing concave order, and the dispersive order; see \cite{muller, shaked2007}.

\smallskip
The paper is structured as follows: In Section~\ref{sec:motivation}, we provide some motivation for maxitivity in decision-making and possibility theory. Section~\ref{sec:basics} introduces the basic concepts and definitions. In Section~\ref{sec:mainResult}, we present and prove the main result. Section~\ref{sec:ex} explores variations of the main result in the context of different stochastic orders. Finally, the appendix provides representation results for maxitive risk measures on \(C_b\), along with auxiliary results related to quantile functions.

\section{Maxitive functions in decision-making}\label{sec:motivation}

Let \( X \) be a set of gambles of the form \( f \colon S \to [-\infty,\infty] \), where each gamble $f$ assigns a loss \( f(s) \) to each possible state \( s \in S \).\footnote{Here, we view negative losses as gains. We note that decision functions are usually defined on spaces of positions. However, for our purposes, describing them in terms of losses seems to be more transparent from a mathematical perspective. Since positions and losses only differ by their sign, all results can be easily transformed.} We consider a decision maker who evaluates the gambles using a decision function \( \psi \colon X \to (-\infty, \infty] \), which in risk analysis is commonly referred to as a \emph{risk measure}; see \cite{follmer}. We adopt the convention that \( \psi(f) \le \psi(g) \) indicates that the gamble \( g \) is considered riskier than the gamble \( f \). The set of acceptable gambles for the decision maker is defined by the \emph{acceptance set}, given by
\(A := \{ f \in X \colon \psi(f) \le 0 \}\).

\smallskip
We assume that the decision maker believes in an ordering $\preccurlyeq$, thus evaluating gambles $f,g\in X$ satisfying $f\preccurlyeq g$ with $\psi(f)\le\psi(g)$. This could be the \emph{pointwise order}, in which case
$f\preccurlyeq_{\rm pw} g$ if and only if $f(s)\le g(s)$ for all $s\in S$.  
In decision-making, however, there are numerous methods for comparing gambles beyond pointwise ordering. These include comparisons based on quantiles, the spread of quantiles, and other distributional properties. Such considerations lead to various stochastic orders, including the \emph{usual stochastic order}, the \emph{increasing convex stochastic order}, or the \emph{dispersive order}. In the first case, we assign probabilities to the states in $S$ and order gambles $f,g\in X$ using their quantile functions, i.e., $f\preccurlyeq_{\rm st} g$ if and only if $q_f(u)\le q_g(u)$ for all $u\in (0,1)$ (or equivalently, if the cumulative distribution function of $f$ is pointwise greater than the cumulative distribution function of $g$).  For further examples of stochastic orders and their definitions, we refer to the monographs \cite{muller,shaked2007}. 

\smallskip
In this article, we focus on decision functions whose acceptance sets remain invariant under \emph{worst-case aggregation}. To be more precise, consider a decision maker with an acceptance set \( A \). For two acceptable positions \( f, g \in A \), we require that the supremum \( f \vee g \) with respect to the decision maker's order \( \preccurlyeq \) exists and remains acceptable, i.e., \( f \vee g \in A \). Equivalently, if \( f \vee g \) is not acceptable, then either \( f \) or \( g \) is not acceptable. Intuitively, this means that losses in certain states cannot be compensated by gains in other states. 
Suppose, in addition, that the decision function \( \psi \)  is monotone, i.e., \( f \preccurlyeq g \) implies \(\psi(f)\le\psi(g)\), and satisfies the translation property, i.e., \( \psi(f + c) = \psi(f) + c \) for all \( f \in X \) and \( c \in \mathbb{R} \), where \( f + c \) denotes the function \( f \) shifted by a constant \( c \in \mathbb{R} \). We then obtain that \( \psi \) is \textit{maxitive}, i.e.,
\[
\psi(f\vee g)=\max\{\psi(f),\psi(g)\}\quad\mbox{for all }f,g\in X.\footnote{ Indeed, since both \( f \) and \( g \) are smaller than \( f \vee g \), we obtain the inequality \( \geq \) by monotonicity. As for the other inequality \( \leq \), let  
\(s := \max\{\psi(f), \psi(g)\}\)  
so that \( f - s, g - s \in A \) due to the translation property, which implies  
\(\psi(f \vee g) \leq s\)  
as the supremum of \( f - s \) and \( g - s \) remains acceptable.
}
\]
We also want to mention that in the literature, there are different versions of maxitivity, 
all of which require \( \psi\left( \vee_{i \in I} f_i \right) = \sup_{i \in I} \psi(f_i) \), but for different families \( (f_i)_{i \in I}\), for which the supremum \( \vee_{i \in I} f_i \) is assumed to exist; see e.g.~\cite{cattaneo,shilkret}. For example, \( \psi \) is called \emph{completely maxitive} if this holds for all families \( (f_i)_{i \in I} \), and \( \psi \) is called \emph{countably maxitive} if this holds for countable families \( (f_n)_{n \in \mathbb{N}} \). Furthermore, a topological perspective on maxitivity, which is closely linked to the large deviation principle, was recently investigated in \cite{kupper2}; see also \cite{kupper,puhalskii}.

\subsection{Maxitivity for the pointwise order in the context of \emph{possibility theory}; see~\cite{dubois,dubois2,zadeh}} As a simple example, consider a loss \( f \) that depends on three possible scenarios \( S = \{\text{good, average, bad}\} \), each with a certain plausibility given by \( \pi \colon S \to [0,1] \). The weighted worst-case criterion then leads to the completely maxitive function  
\[\psi(f) = \max_{s \in S} f(s) \pi(s)\] on the space of positive gambles \( f \colon S \to [0, \infty] \). We note that this example generalizes to arbitrary state spaces \( S \). Moreover, the mapping \( \sup_{s \in S} f(s) \pi(s) \) coincides with the \emph{Shilkret integral} \( \int^S f \, \mathrm{d} \Pi \) with respect to the possibility measure \( \Pi \), which assigns each subset \( B \subset S \) the value \( \Pi(B) := \sup_{s \in B} \pi(s) \), where 
\( \pi \) is a possibility distribution; see e.g.~\cite{cattaneo,cooman,mesiar,shilkret}. 
Another closely related example is the penalized worst-case evaluation \( \sup_{s \in S} \{ f(s) - \alpha(s) \} \) for some penalty \( \alpha \colon S \to [-\infty, \infty] \). Again this is a completely maxitive function on the space \( X \) of all gambles \( f \colon S \to [-\infty, \infty] \). Conversely, it is known  that every completely maxitive function \( \psi \colon X \to [-\infty, \infty] \), which has the translation property and satisfies \( \psi(0) \in \mathbb{R} \), admits the representation
\begin{equation}\label{eq:rep pw}
\psi(f) = \sup_{s \in S} \{ f(s) - \alpha(s) \} \quad \text{for all } f \in X,
\end{equation}
for some penalty \( \alpha \colon S \to [-\infty, \infty] \); see~\cite[Corollary 7]{cattaneo}. The right-hand side in the previous equation corresponds to the so-called convex integral, which is obtained as a transformation of the Shilkret integral; see \cite{cattaneo}.  
In the translation-invariant case (particularly for monetary risk measures), these considerations show that maxitive functions with respect to the pointwise order take the form~\eqref{eq:rep pw} of a penalized worst-case evaluation over the state space.

\subsection{Maxitivity with respect to other orders} As an illustration, we start with the usual stochastic order. In this case, we order the distributions of the gambles according to their quantile functions. We recall that for two gambles \( f \) and \( g \), there exists a gamble \( f \vee_{\rm st} g \) whose quantile function is the smallest one that dominates the quantiles \( q_f \) and \( q_g \); see~\cite{muller}. A decision function is then maxitive if 
\( \psi(f \vee_{\rm st} g) = \max\{\psi(f), \psi(g)\} \) for all gambles $f,g$.
Intuitively, maxitivity still ensures that losses in certain states cannot be compensated by gains in other states. However, in this context, the relevant `states' refer to the levels of the quantile functions.  More formally, suppose we have two gambles \( f \) and \( g \), where \( f \) is acceptable but \( f \vee_{\rm st} g \) is not acceptable, for a decision function \( \psi \) with acceptance set \( A \). If \( \psi \) is maxitive, then any gamble \( h \) that satisfies \( q_s(h) = q_s(g) \) for all \( s \) in the ‘bad states’ \( U = \{s \colon q_g(s) \geq q_f(s)\} \) is not acceptable, regardless of its values \( q_s(h) \) in the ‘good states’ \( s \in U^c \). Indeed, since \( f \vee_{\rm st} g \preccurlyeq f \vee_{\rm st} h \), it follows that  
\[
0 < \psi(f \vee_{\rm st} g) \leq \psi(f \vee_{\rm st} h) = \max\{\psi(f), \psi(h)\}.
\]  
Given that \( \psi(f) \leq 0 \), this implies \( \psi(h) > 0 \), meaning that \( h \) is not acceptable. 
In other words, maxitivity with respect to the usual stochastic order implies that risky quantiles at certain levels  (`bad states') cannot be compensated by less risky quantiles at other levels  (`good states'). 
When dealing with probability distributions, it seems natural to consider quantities (that are closely linked to the distribution) as the relevant `states'. 
 A typical example of such a maxitive function is the \emph{Value at Risk}, which is widely used in the financial industry. Formally, the Value at Risk at level \( s \in (0,1) \), defined by \( {\rm VaR}_s(f) := q_f(s) \), is maxitive with respect to the usual stochastic order but is not maxitive with respect to the pointwise order; see~\cite{wang}.
 
 In the context of model aggregation in risk analysis, an interesting motivation for maxitivity with respect to the usual stochastic order was recently presented in~\cite{wang}, where the topic is discussed in great detail. Simplifying their example, consider a risk analyst who evaluates random losses using different methods, resulting in a set of possible risk models \( \{ f_1, \dots, f_n \} \), which are aggregated into a worst-case model \( f^\ast := f_1 \vee_{\rm st} \cdots \vee_{\rm st} f_n \). Under maxitivity, we obtain the desirable property that the risk of the worst-case model \( \psi(f^\ast) \) equals the worst-case risk \( \max\{\psi(f_1), \dots, \psi(f_n)\} \). Moreover, it is shown in \cite{wang} that every real-valued function $\psi$, which has the translation property,  is lower semicontinuous  with respect to the convergence in distribution\footnote{I.e., $\liminf_{n\to\infty} \psi(f_n)\ge \psi(f)$ whenever $f_n$ converges to $f$ in distribution.} and is completely maxitive (with respect the usual stochastic order) admits a representation
\begin{equation}\label{eq:rep st}
 \psi(f)=\sup_{s\in(0,1)}\{q_f(s)-\alpha(s)\}\quad\mbox{ for all }f,
\end{equation}
 for some penalty $\alpha\colon (0,1)\to[-\infty,\infty]$. 
 The representations in \eqref{eq:rep st} and \eqref{eq:rep pw} are closely related, as \eqref{eq:rep st} can be interpreted as an analogue of \eqref{eq:rep pw}, where quantile evaluation replaces scenario-wise evaluation.

  \smallskip
  Another widely used preorder in decision-making is the \emph{dispersive order}, which is employed to compare the dispersion or variability of two gambles. Specifically, for two gambles \( f, g \), we write \( f \preccurlyeq_{\rm disp} g \) if and only if  
 \[
 q_f(v) - q_f(u) \leq q_g(v) - q_g(u) \quad \text{whenever } 0 < u < v < 1.
 \]  
 For any two gambles \( f \) and \( g \), there exists a gamble \( f \vee_{\rm disp} g \), which is the least upper bound of \( f \) and \( g \) with respect to \( \preccurlyeq_{\rm disp} \). Intuitively, \( f \vee_{\rm disp} g \) is the least dispersed gamble that is more dispersed than both \( f \) and \( g \); see \cite{muller}. We say that a function \( \psi \colon X \to \mathbb{R} \) is maxitive with respect to \( \preccurlyeq_{\rm disp} \) if \( \psi(f \vee_{\rm disp} g) = \max\{\psi(f), \psi(g)\} \) for any two gambles $f,g$.  In this context, maxitivity once again ensures that losses in certain states cannot be offset by gains in others. Here, the relevant `states' correspond to all pairs of levels \( (u,v) \) with \( 0 < u < v < 1 \), at which the quantile spread \( q_g(v) - q_g(u) \) is evaluated. Thereby, maxitivity with respect to the dispersive order implies that risky quantile spreads at certain pairs of levels cannot be compensated by less risky quantile spreads at other levels.

\smallskip
These ideas naturally extend to other partial orders. 
Consider a decision-maker who compares gambles based on the evaluations of certain relevant features represented by a set \(\Phi\) of functions, i.e.,  \(f \preccurlyeq g\) if and only if \(\phi(f) \leq \phi(g)\) for all \(\phi \in \Phi\).\footnote{For the pointwise order, we consider the evaluation \(\phi_s(f) = f(s)\) for each state \(s \in S\). For the usual stochastic order, the evaluation is given by \(\phi_s(f) = q_f(s)\) at each quantile level \(s \in (0,1)\). For the dispersive order, we use the evaluation \(\phi_{(u,v)}(f) = q_f(v) - q_f(u)\) for each pair \((u,v) \in \Delta\) of quantile levels.}   
In this situation, assuming that any two gambles have a least upper bound, a decision function \( \psi\colon X\to\mathbb{R} \) is called \textit{maxitive} if it satisfies  
\(\psi(f\vee_{\Phi} g) = \max\{\psi(f), \psi(g)\}\)
for all gambles \( f, g \). 
We prove in this paper that, in the translation-invariant case and under reasonable assumptions,  maxitive functions can be represented in the form 
\begin{equation}\label{eq:rep general}
\psi(f) = \sup_{\phi \in \Phi} \{\phi(f) - \alpha(\phi)\},
\end{equation}  
for some penalty function \(\alpha \colon \Phi \to [-\infty, \infty]\).  However, in general, a function that admits a representation \eqref{eq:rep general} is not necessarily maxitive. To ensure this, the penalty has to satisfy an additional stability condition. Furthermore, we want to emphasize that the \(\phi\)'s generally do not satisfy the translation property (e.g.~for the dispersive order), in which case the representation in \eqref{eq:rep general} is more involved.

\section{Basic concepts and definitions}\label{sec:basics}

In this section, we provide the formal definitions of \emph{maxitivity with respect to general orders} and introduce the relevant related concepts. Let \( X \) be a non-empty set\footnote{In the previous motivation, \(X\) was a set of gambles.}, which is endowed with a preorder
\[
x \preccurlyeq y \quad \text{if and only if} \quad \phi(x) \le \phi(y) \quad \text{for all } \phi \in \Phi,
\]
where \( \Phi \) is a non-empty set of functions \( \phi \colon X \to \mathbb{R} \).  Throughout this article, the notation \( Y \subset X \) denotes non-strict inclusion, meaning that every element of \( Y \) is also an element of \( X \), allowing for the possibility that \( Y = X \).

\smallskip
For the moment, we assume that $(X,\preccurlyeq)$ is an upper semilattice, i.e., for any two elements $x,y\in X$, the set $\{x,y\}$ has a unique supremum $x\vee y\in X$. 
In that case, the preoder $\preccurlyeq$ is a partial order.\footnote{Indeed, for $x,y\in X$ with $x \preccurlyeq y$ and $y \preccurlyeq x$, there exists a unique sumpremum $x\vee y$ satisfying $x \preccurlyeq x\vee y$ and $y \preccurlyeq x\vee y$. Since $x,y$ are upper bounds for $\{x,y\}$, we conclude that $x=y=x\vee y$.} Our focus is on functions \(\psi \colon X \to (-\infty, \infty]\) that are maxitive with respect to the order \(\preccurlyeq\), i.e., \[\psi(x \vee y) = \max\{\psi(x), \psi(y)\}\quad\mbox{for all }x, y \in X.\] Note that the supremum, and thus maxitivity, depends on the order \( \preccurlyeq \). Such maxitive functions are characterized as follows:
\begin{lemma}\label{lem:finiteMax}
Let $(X,\preccurlyeq)$ be an upper semilattice. For a function  \(\psi \colon X \to (-\infty, \infty]\), 
the following conditions are equivalent:
\begin{itemize}
    \item[(i)] $\psi$ is maxitive.
    \item[(ii)] $\psi$ is monotone\footnote{I.e., $x\preccurlyeq y$ implies $\psi(x)\le \psi(y)$.} and has upwards directed sublevel sets\footnote{I.e., if $x,y\in A_s=\{z\in X\colon \psi(z)\le s\}$ for some $s\in\R$, then there exists $z^\ast\in A_s$ with $x \preccurlyeq z^\ast$ and $y \preccurlyeq z^\ast$.}.
    \item[(iii)] $\psi$ is monotone and for every partition $\{\Phi_1,\Phi_2\}$ of $\Phi$ and $x_1,x_2\in X$ with  $\phi(x_1)\ge\phi(x_2)$ for all $\phi\in\Phi_1$ and  $\phi(x_2)\ge\phi(x_1)$ for all $\phi\in\Phi_2$, there exists $i\in \{1,2\}$ such that
    \[
\phi(y)\ge \phi(x_i)\mbox{ for all }\phi\in\Phi_i\quad\mbox{implies}\quad
    \psi(y)\ge \psi(x_1\vee x_2).\]
\end{itemize}
\end{lemma}
\begin{proof}
\({\rm (i)} \Rightarrow {\rm (ii)}\): Given $x,y\in X$ with $x \preccurlyeq y$, we have that $y=x\vee y$. 
Since $\psi$ is maxitive, we obtain that $\psi(y)=\psi(x\vee y)=\max\{\psi(x),\psi(y)\}\ge \psi(x)$, which shows that $\psi$ is monotone. Moreover, given $x,y\in A_s$ for some $s\in\R$, we fix $z^\ast=x\vee y$. By maxitivity, it follows that $\psi(z^\ast)=\psi(x)\vee \psi(y)\le s$, so that $z^\ast\in A_s$. This shows that $\psi$ has upwards directed sublevel sets.

\({\rm (ii)} \Rightarrow {\rm (i)}\): Given $x,y\in X$, by monotonicity, we have that $\psi(x\vee y)\ge \max\{\psi(x),\psi(y)\}$. Suppose by contradiction that $\psi(x\vee y)>s:=\max\{\psi(x),\psi(y)\}$.  Since $x,y\in A_s$ and $\psi$ has upwards directed sublevel sets, there exists $z^\ast\in A_s$ such that $x,y \preccurlyeq z^\ast$. Hence, $x\vee y  \preccurlyeq z^\ast$, which by monotonicity implies that $\psi(x\vee y)\le \psi(z^\ast)\le s$. This contradicts that $\psi(x\vee y)>s$.

 $\rm{(i)}\Rightarrow{\rm (iii)}$: By contradiction, assume that there exist a partition  $\{\Phi_1,\Phi_2\}$ of $\Phi$ and $x_1,x_2\in X$ with  $\phi(x_1)\ge\phi(x_2)$ for all $\phi\in\Phi_1$ and $\phi(x_2)\ge\phi(x_1)$ for all $\phi\in\Phi_2$, for which there exist $y_1,y_2\in X$ with 
 \[
 \phi(y_i)\ge \phi(x_i)\mbox{ for all }\phi\in\Phi_i\quad \mbox{and}\quad\psi(y_i)<\psi(x_1\vee x_2),
 \]
 for $i\in\{1,2\}$.
Then, for any \(\phi \in \Phi\), either \(\phi \in \Phi_1\) so that \(\phi(y_1 \vee y_2) \ge \phi(x_1) \ge \phi(x_2)\), or \(\phi \in \Phi_2\) so that \(\phi(y_1 \vee y_2) \ge \phi(x_2) \ge \phi(x_1)\), which shows that \(x_1 \vee x_2 \preccurlyeq y_1 \vee y_2\). Since \(\psi\) is maxitive, it is also monotone, which leads to the following contradiction:
\[
\psi(x_1 \vee x_2) \le \psi(y_1 \vee y_2) = \max\{\psi(y_1), \psi(y_2)\} <   \psi(x_1 \vee x_2).
\]

 ${\rm (iii)}\Rightarrow {\rm (i)}$:  For $x_1,x_2\in X$, we consider the partition $\Phi_1:=\{\phi\colon \phi(x_1)\ge \phi(x_2)\}$ and $\Phi_2:=\{\phi\colon \phi(x_1)< \phi(x_2)\}$. Then, there exists $i\in \{1,2\}$ such that
     \[
     \phi(y)\ge \phi(x_i)\mbox{ for all }\phi\in\Phi_i\quad\mbox{implies}\quad
    \psi(y)\ge \psi(x_1\vee x_2).
     \]
   By choosing \( y := x_i \) in the previous relation, we obtain that \( \psi(x_1 \vee x_2) \le \psi(x_i) \le \max\{\psi(x_1),\psi(x_2)\} \). The other inequality \( \psi(x_1 \vee x_2) \ge \max\{\psi(x_1),\psi(x_2)\} \) follows directly from the monotonicity of \( \psi \).
\end{proof}
\begin{discussion}
We note that our concept includes the classical notion of maxitivity. To see this, let \( X \) be the space of all functions \( f \colon S \to [-\infty, \infty] \) on some state space \(S\), and consider the family $\Phi:=\{\phi_s\colon s\in S\}$, where $\phi_s\colon X\to \mathbb{R}$ is given by $\phi_s(f)=\arctan(f(s))$.\footnote{We set $\arctan(\infty):=\pi/2$ and $\arctan(-\infty):=-\pi/2$.} In this case, \((f \vee g)(s) = \max\{f(s), g(s)\}\), and the concept of maxitivity corresponds to the usual notion of maxitivity studied in the literature; see e.g.~\cite{cattaneo,cooman,maslov,shilkret}. Moreover, the state space \(S\) can be identified with the space \(\Phi\), which defines the order \(\preccurlyeq\). In line with the motivation presented in the previous section, Lemma~\ref{lem:finiteMax}(iii) guarantees that losses in certain states cannot be compensated by gains in others. For a general order, this result remains valid when interpreting the elements of \(\Phi\) as the relevant states.
\end{discussion}

\smallskip
Representation results in terms of a weighted or a penalized worst-case evaluation for maxitive functions generally require stronger forms of maxitivity; see  e.g.~\cite{cattaneo,cooman}. For instance, \(\psi\) is called completely maxitive if 
\[
\psi\big(\vee_{i \in I} x_i\big) = \sup_{i \in I} \psi\left(x_i\right)
\]
for every family \((x_i)_{i \in I}\subset X\). 
However, for such a notion to make sense, the supremum \(\vee_{i \in I} x_i\) must exist for arbitrary families. Under additional regularity assumptions, complete maxitivity is already ensured by (finite) maxitivity  \cite{arslanov,kupper,kupper2,murofushi}. 
For our purposes, and to achieve the most elegant mathematical representation results, we will equip the `state space' \(\Phi\) with a Polish topology and assume that the mapping \(\Phi \to \mathbb{R}\), \(\phi \mapsto \phi(x)\), is lower semicontinuous for all $x\in X$; see Section~\ref{sec:mainResult}.
In this context, we work with the following version of maxitivity:
\begin{definition}\label{def:countmax}
Suppose that  $(X,\preccurlyeq)$ is countably complete\footnote{We say that \((X, \preccurlyeq)\) is countably complete if every countable subset\footnote{ Recall taht we use the convention that $A\subset B$ means that every element of $A$ is an element of $B$.} \(Y \subset X\) that is bounded from above has a unique supremum \(\vee Y\) in \(X\). Recall that \(Y\) is bounded from above, if there exists $y^\ast\in X$ such that $y\preccurlyeq y^\ast$ for all $y\in Y$.}. Then, a function $\psi\colon X\to(-\infty,\infty]$ is \emph{bounded countably maxitive} if for every countable subset $Y\subset X$, we have 
\begin{equation}\label{eq:YMaxitive}
\sup_{y\in Y}\psi(y)
=
\begin{cases}
\psi(\vee Y), & \mbox{ if }Y\mbox{ is bounded from above,}\\
\infty, & \mbox{ else.}
\end{cases}
\end{equation}
\end{definition}
As in the finite case described in Lemma~\ref{lem:finiteMax}, we obtain the following characterization of bounded countable maxitivity:
\begin{lemma}\label{relation:maxitive}
Suppose that \((X, \preccurlyeq)\) is countably complete. Then, a function \(\psi \colon X \to (-\infty, \infty]\) is bounded countably maxitive if and only if it is monotone and has countably upwards directed sublevel sets\footnote{A function \(\psi\) has countably upwards directed sublevel sets if for every \(s \in \mathbb{R}\), the level set \(A_s := \{x \in X : \psi(x) \le s\}\) is countably upwards directed, i.e., for every countable subset  \(Y \subset A_s\), there exists \(y^\ast \in A_s\) with \(y \preccurlyeq y^\ast\) for all \(y \in Y\).}.
\end{lemma}
\begin{proof}
First, suppose that $\psi$ is bounded countably maxitive. Then, for $x,y\in X$ with $x\preccurlyeq y$, it follows from bounded countably maxitivity with $Y:=\{x,y\}$ that $\psi(x)\le \psi(x\vee y)=\psi(y)$. Moreover, for any countable subset $Y\subset A_s$, it holds
$\sup_{y\in Y} \psi(y)\le s$ and bounded countable maxitivity implies that $Y$ is bounded from above with $\vee Y\in A_s$.

Second, if $\psi$ is monotone, then the right hand side of \eqref{eq:YMaxitive} is greater than the left hand side. If  $\psi$ has countably upwards directed sublevel sets, then either $\sup_{y\in Y}\psi(y)=\infty$ or $\sup_{y\in Y}\psi(y)=s<\infty$ in which case $\vee Y\in A_s$ so that $\psi(\vee Y)\le s$, i.e., the right hand side of \eqref{eq:YMaxitive} is smaller than the left hand side. 
\end{proof}

\smallskip
Our representation results are based on duality arguments. Hereby, maxitivity translates dually into the following stability property.

\begin{definition}\label{def:stability}
    A function $\alpha\colon {\Phi}\to [-\infty,\infty]$ is  \emph{countably stable} if for every countable subset $Y\subset X$ with  $\sup_{y\in Y}\phi(y)\le \alpha(\phi)$ for all $\phi\in\Phi$,
there exists $y^\ast\in X$ with $y\preccurlyeq y^\ast$ for all $y\in Y$ and $\phi(y^\ast)\le \alpha(\phi)$ for all $\phi\in\Phi$.
\end{definition}

At the level of sublevel sets, we have the following characterization.

\begin{proposition}\label{prop:duality}
 Suppose that 
\[
A = \{x \in X \colon \phi(x) \le \alpha(\phi) \text{ for all } \phi \in \Phi\},
\]
for some \(\alpha \colon \Phi \to [-\infty, \infty]\). Then, we have  
\[
A = \{x \in X \colon \phi(x) \le \alpha_{\min}(\phi) \text{ for all } \phi \in \Phi\},
\]
where \(\alpha_{\min}(\phi) := \sup_{x \in A} \phi(x)\) for all \(\phi \in \Phi\).

In this case, \(A\) is countably upwards directed if and only if \(\alpha\) is countably stable if and only if \(\alpha_{\min}\) is countably stable.
\end{proposition}

\begin{proof}
First, since $\phi(x)\le\alpha(\phi)$ for all $x\in A$ and $\phi\in\Phi$, it holds that $\alpha_{\min}(\phi)\le \alpha(\phi)$ for all $\phi\in\Phi$, which shows that $\{x\in X\colon \phi(x)\le \alpha_{\min}(\phi)\mbox{ for all }\phi\in\Phi\}\subset A$. Conversely, if $x\in A$, then $\phi(x)\le\alpha_{\min}(\phi)$ for all $\phi\in\Phi$, which shows that $A\subset\{x\in X\colon \phi(x)\le \alpha_{\min}(\phi)\mbox{ for all }\phi\in\Phi\}$.

Second, if \(A\) is countably upwards directed and the subset \(Y \subset X\) is countable with \(\sup_{y \in Y} \phi(y) \le \alpha(\phi)\) for all \(\phi \in \Phi\), then \(Y \subset A\), so there exists \(y^* \in A\) with \(y \preccurlyeq y^*\) for all \(y \in Y\), and therefore \(\phi(y^*) \le \alpha(\phi)\) for all \(\phi \in \Phi\), i.e., \(\alpha\) is countably stable. Conversely, if \(\alpha\) is countably stable and \(Y \subset X\) is countable such that \(Y \subset A\), then \(\sup_{y \in Y} \phi(y) \le \alpha(\phi)\) for all \(\phi \in \Phi\), and countable stability implies that \(Y\) is bounded from above by some \(y^* \in A\), i.e., \(A\) is countably upwards directed. The argument for \(\alpha_{\min}\) is analogous.
\end{proof}

\begin{remark}
In general, for an arbitrary countably upwards directed subset \( A \subset X \), it does not follow that \( \alpha_{\min}(\phi) := \sup_{x \in A} \phi(x) \) is countably stable, even if \( X \) is countably complete.

To give a counterexample, let \( X \) be the set of all functions \( f \colon \mathbb{R} \to [0,1] \) such that there exists a bounded interval \( I \subset \mathbb{R} \) for which \( \{x \in \mathbb{R} \setminus I \colon f(x) > 0 \} \) is at most countable. We consider the family \( \Phi := \{ \phi_x \colon x \in \mathbb{R} \} \), where \( \phi_x(f) = f(x) \). The ordered set \( (X, \preccurlyeq) \) is countably complete. 
Now, let \( A \) be the set of all functions \( f \in X \) such that \( \{x \in \mathbb{R} \colon f(x) > 0 \} \) is at most countable. Then, the set \( A \) is countably upwards directed, and 
\[
\alpha_{\min}(\phi_x) = \sup_{f \in A} f(x) = 1 \quad \text{for all } x \in \mathbb{R}.
\]
To show that \( \alpha_{\min} \) is not countably stable, consider for every \( n \in \mathbb{N} \) the indicator function \( 1_{[-n,n]} \) which assumes the value \( 1 \) on \( [-n,n] \) and \( 0 \) on \( [-n,n]^c \). Then, it holds 
\[
\sup_{n \in \mathbb{N}} 1_{[-n,n]}(x) \le \alpha_{\min}(\phi_x) \quad \text{for all } x \in \mathbb{R},
\]
but there is no function \( f \in X \) with \( 1_{[-n,n]}(x) \le f(x) \) for all \( x \in \mathbb{R} \).
\end{remark}

\section{Main result} \label{sec:mainResult}

In this section, we present our representation results for maxitive functions with respect to general orders. Since we do not assume any form of completeness on the order, we consider monotone functions with countably upwards directed sublevel sets.  In contrast to bounded countable maxitivity, these concepts do not require that the order is countably complete, allowing us to derive more general statements without relying on the assumption of countable completeness.

Throughout this section, we assume that \( X \) is a non-empty set, and \( \Phi \) is a non-empty set of functions \( \phi \colon X \to \mathbb{R} \) which induces the preorder 
\[
x \preccurlyeq y \quad \text{if and only if} \quad \phi(x) \le \phi(y) \text{ for all } \phi \in \Phi.
\] 
In addition, we assume that the set \( \Phi \) is endowed with a topology such that \( \Phi \) is a Polish space, and the functions \( f_x \colon \Phi \to \mathbb{R} \) defined by \( f_x(\phi) := \phi(x) \) are lower semicontinuous\footnote{I.e., $\{\phi\in\Phi\colon f_x(\phi)\le s\}$ is closed for all $s\in\R$.} for all \( x \in X \).

The following is the main result of this article.
\begin{theorem}\label{thm:main}
Let \( \psi \colon X \to (-\infty, \infty] \) be a function with sublevel sets \( (A_s)_{s \in \mathbb{R}} \) given by
\[
A_s := \{x \in X \colon \psi(x) \le s\}.
\]
Then, the following statements are equivalent:
\begin{itemize}
\item[(i)] $\psi$ is monotone and has countably upwards directed sublevel sets.
\item[(ii)] There exists a function $\alpha\colon \mathbb{R}\times \Phi\to[-\infty,\infty]$ which is non-decreasing in the first argument and countably stable in the second argument such that 
\begin{equation}\label{eq:mainRep}
\psi(x)=\inf\left\{s \in\R \colon \underset{\phi\in\Phi}\sup\{\phi(x)-\alpha(s,\phi)\} \le 0\right\}\quad\mbox{for all }x\in X.
\end{equation}
The infimum is attained for all $x\in X$ with $\psi(x)<\infty$.

\item[(iii)]  The function $\alpha_{\min}\colon\mathbb{R}\times\Phi\to [-\infty,\infty]$, $\alpha_{\min}(s,\phi):=\sup_{x\in A_s} \phi(x)$ is countably stable in the second argument and
\begin{equation}\label{eq:mainRepBis}
\psi(x)=\inf\left\{s \in \R\colon \underset{\phi\in\Phi}\sup\{\phi(x)-\alpha_{\min}(s,\phi)\} \le 0\right\}\quad\mbox{for all }x\in X.
\end{equation}
The infimum is attained for all $x\in X$ with $\psi(x)<\infty$.

\item[(iv)]  The function $\alpha^+_{\min}\colon\mathbb{R}\times\Phi\to [-\infty,\infty]$, $\alpha_{\min}^+(s,\phi):=\inf_{s<t}\alpha_{\min}(t,\phi)$ is countably stable in the second argument and
\begin{equation}\label{eq:mainRepII}
\psi(x)=\inf\left\{s \in \R\colon \underset{\phi\in\Phi}\sup\{\phi(x)-\alpha^+_{\min}(s,\phi)\} \le 0\right\}\quad\mbox{for all }x\in X.
\end{equation}
The infimum is attained for all $x\in X$ with $\psi(x)<\infty$.

\item[(v)] The function $\alpha_{\min}^+$ is countably stable in the second argument and 
\begin{equation}\label{eq:mainRepIII}
\psi(x)=\sup_{\phi\in \Phi} G(\phi(x),\phi)\quad\mbox{for all }x\in X,
\end{equation}
where $G\colon\mathbb{R}\times\Phi\to [-\infty,\infty]$, $G(t,\phi):=\inf\{s\in\mathbb{R}\colon t\le \alpha^+_{\min}(s,\phi)\}$.
\end{itemize}
\end{theorem}

\begin{proof}
Let $L(\Phi)$ denote the set of all lower semicontinuous functions $f\colon \Phi\to \mathbb{R}$ endowed with the pointwise order. Recall that $f_x\colon \Phi\to\mathbb{R}$ is defined by $f_x(\phi):=\phi(x)$ for all $x\in X$. Moreover, let $A_s:=\{x\in X\colon \psi(x)\le s\}$ for all $s\in \mathbb{R}$.

 $(i)\Rightarrow(ii)$: We can assume w.l.o.g. that  $\inf_{\phi\in\Phi}\phi(x)>-\infty$ for all $x\in X$.\footnote{Otherwise, we prove (ii) for the set $\Phi^\prime:=\{e^{\phi}\colon \phi\in \Phi\}$, which also induces the preorder $\preccurlyeq$. 
 In fact, if (ii) holds for the set $\Phi^\prime$ and some function $\alpha^\prime\colon \R\times\Phi^\prime\to [-\infty,\infty]$, then it follows by inspection that (ii) is also satisfied by the original family $\Phi$ and the function $\alpha\colon \R\times\Phi\to [-\infty,\infty]$ defined by $\alpha(s,\phi):=\log(0\vee \alpha^\prime(s,e^{\phi}))$.}
  For $I:=\{s\in\mathbb{R}\colon A_s\neq\emptyset\}$, we define
 \begin{equation}\label{eq:defF}
 F(s,f)
 :=\inf_{x\in A_s}\sup_{\phi\in\Phi}\{f(\phi)-\phi(x)\}\quad\mbox{for all } s\in I\mbox{ and }f\in L(\Phi).
 \end{equation}
 Then, the mapping $F\colon I\times L(\Phi)\to [-\infty,\infty]$ is well-defined and satisfies the following properties. For every $f,g\in L(\Phi)$ and $s,t\in I$, it holds
 \begin{itemize}
 \item[(1)] $F(s,f) \le F(t,f)$ for all $t\le s$,
 \item[(2)] $F(s,f)\le F(s,g)$ whenever $f\le g$,
 \item[(3)] $F(s,f+c)=F(s,f)+c$ for all $c\in\mathbb{R}$,
 \item[(4)] $F(s,\sup_{n\in\mathbb{N}} f_n)=\sup_{n\in \mathbb{N}}F(s,f_n)$ for every sequence $(f_n)_{n\in\mathbb{N}}\subset L(\Phi)$ with $\sup_{n\in\N} f_n\in L(\Phi)$,
 \item[(5)] $F(s,f)\in (-\infty,\infty]$ and $F(s,0)\in\mathbb{R}$, 
  \item[(6)] $A_s=\{x\in X\colon F(s,f_x)\le 0\}$,
  \item[(7)] $\psi(x)=\inf\{s\in I\colon  F(s,f_x)\le 0\}$ for all $x\in X$, and the infimum is attained whenever $\psi(x)<\infty$. 
 \end{itemize}
The properties (1)--(3) follow directly from the definition of $F$.

To show (4), let $(f_n)_{n\in\mathbb{N}}\subset L(\Phi)$ with $f:=\sup_{n\in\N} f_n\in L(\Phi)$. Then, it follows from condition (2) that $\sup_{n\in\N} F(s,f_n)\le F(s,f)$.
By contradiction, we assume there exists $c\in\mathbb{R}$ with 
\[
\sup_{n\in\N} F(s,f_n)<c< F(s,f).
\]
Then for every $n\in\N$ there exists $x_n\in A_s$ with
$f_n(\phi)-\phi(x_n)<c$ for all $\phi\in\Phi$. Since $A_s$ is countably upwards directed, there exists $x^\ast\in A_s$ with $x_1,x_2,\cdots \preccurlyeq x^\ast$. We obtain that
$f_n(\phi)-\phi(x^\ast)<c$ for all $\phi\in\Phi$ and $n\in\N$.
Taking the supremum first over all $n\in\N$, and second over all $\phi\in\Phi$, we obtain that
\[
\sup_{\phi\in\Phi}\{f(\phi)-\phi(x^\ast)\}\le c<F(s,f)=\inf_{x\in A_s}\sup_{\phi\in\Phi}\{f(\phi)-\phi(x)\},
\]
which is a contradiction to $x^\ast\in A_s$. 

To verify (5), suppose by contradiction that $F(s,f)=-\infty$.
Then, for $\phi\in\Phi$ there exists $(x_n)_{n\in\mathbb{N}}\subset A_s$ with $f(\phi)-\phi(x_n)\le -n$ for all $n\in\mathbb{N}$. Since $A_s$ is countably upwards directed, there is $x^\ast\in A_s$ with $x_1,x_2,\cdots \preccurlyeq x^\ast$, and therefore  
 \[
n+f(\phi)\le \phi(x^\ast)\quad\mbox{ for all }n\in\mathbb{N}.
\] 
Letting $n\to\infty$ yields $\phi(x^\ast)=\infty$ in contrast to the assumption that $\phi$ is real-valued. This shows that $F(s,f)\in(-\infty,\infty]$. In particular, we obtain that for every $x\in A_s$,
\[
-\infty<F(s,0)=\inf_{y\in A_s}\sup_{\phi\in\Phi}\{-\phi(y)\}\le -\inf_{\phi\in\Phi}\phi(x)<\infty,
\]
where the last inequality follows from the boundedness assumption on the family $\Phi$, i.e., $F(s,0)\in\R$. 

Next, we show (6). If $x\in A_s$, then
\begin{align*}
F(s,f_x)&=\inf_{y\in A_{s}}\sup_{\phi\in\Phi}\{\phi(x)-\phi(y)\}\le \sup_{\phi\in\Phi}\{\phi(x)-\phi(x)\}=0.
\end{align*}   
If $x\in X$ satisfies $F(s,f_x)\le 0$, then there exists $(x_n)_{n\in\mathbb{N}}\subset A_s$ with  
$\phi(x)-\phi(x_n)<1/n$ for all $\phi\in\Phi$.
Since $A_s$ is countably upwards directed, there exists $x^\ast\in A_s$ with $x_1,x_2,\cdots \preccurlyeq x^\ast$, and therefore $\phi(x)-\phi(x^\ast)\le 0$ for all $\phi\in\Phi$,  which shows $x\preccurlyeq x^\ast$. Since $\psi$ is monotone, we conclude that 
$\psi(x)\le \psi(x^\ast)\le s$, and therefore $x\in A_s$. 

To show (7), we notice that condition (6) implies that for every $x\in X$,
\[
\psi(x)=\inf\{s\in I\colon x\in A_s\}=\inf\{s\in I\colon F(s,f_x)\le 0\},
\]
and the infimum is attained whenever $\psi(x)<\infty$.
Hence, we have verified the conditions (1)--(7).

Due to conditions (3)--(5), for each $s\in I$, the function 
$F(s,\cdot)\colon L(\Phi)\to (-\infty,\infty]$ satisfies the assumptions of Corollary~\ref{cor:repMax2}, and we obtain the representation
\begin{equation}\label{eq:repMax}
F(s,f)=\underset{\phi\in\Phi}\sup\{f(\phi)-\alpha(s,\phi)\}\quad\mbox{ for all }s\in I,\:f\in L(\Phi),
\end{equation} 
where $\alpha(s,\phi):=
\sup_{f\in C_b(\Phi)}\{f(\phi)-F(s,f)\}$ for all $s\in I$ and $\phi\in \Phi$. As a result of condition (7), we obtain that 
\begin{equation}\label{eq:repMax2}
\psi(x)=\inf\Big\{s\in \mathbb{R} \colon \sup_{\phi\in\Phi}\{\phi(x)-\alpha(s,\phi)\} \le 0\Big\}\quad\mbox{for all }x\in X,
\end{equation} 
where we set $\alpha(s,\phi):=-\infty$ whenever $s\notin I$.
Moreover, the infimum is attained whenever $\psi(x)<\infty$.
Due to condition (1), the mapping $\alpha$ is non-decreasing in the first argument. For $s\in\mathbb{R}$, the representation \eqref{eq:repMax2} implies that $x\in A_s$ if and only if $\phi(x)\le \alpha(s,\phi)$ for all $\phi\in\Phi$. 
Since $A_s$ is countably upwards directed, it follows from Proposition~\ref{prop:duality} that $\alpha$ is countably stable in the second argument. 

$(ii)\Rightarrow(iii)$: Fix $s\in \mathbb{R}$. The representation~\eqref{eq:mainRep} implies that $x\in A_s$ if and only if $\sup_{\phi\in\Phi}\{\phi(x)-\alpha(s,\phi)\} \le 0$. Since $\alpha$ is countably stable in the second argument, it follows from Proposition~\ref{prop:duality} that $A_s$ is countably upwards directed.
Moreover, Proposition~\ref{prop:duality} implies that $x\in A_s$ if and only if $\sup_{\phi\in\Phi}\{\phi(x)-\alpha_{\min}(s,\phi)\} \le 0$, where 
$\alpha_{\min}(s,\phi):=\sup_{x\in A_s} \phi(x)$, and $\alpha_{\min}$ is countably stable in the second argument.

$(iii)\Rightarrow(iv)$: Fix $s\in\mathbb{R}$. It follows from (iii) together with Proposition~\ref{prop:duality} that $A_s$ is countably directed and $x\in A_s$ if and only if $\sup_{\phi\in\Phi}\{\phi(x)-\alpha_{\min}(s,\phi)\} \le 0$. Hence, $x\in X$ satisfies 
$\sup_{\phi\in\Phi}\{\phi(x)-\alpha_{\min}(s,\phi)\} \le 0$, if and only if $x\in A_t$ for all $t>s$, if and only if $\sup_{t>s}\sup_{\phi\in\Phi}\{\phi(x)-\alpha_{\min}(t,\phi)\} \le 0$, if and only if $\sup_{\phi\in\Phi}\{\phi(x)-\alpha^+_{\min}(t,\phi)\} \le 0$. In particular, it follows from Proposition~\ref{prop:duality} that $\alpha^+$ is countable stable in the second argument.

$(iv)\Rightarrow(v)$: Since $\alpha_{\min}^+$ is right-continuous and non-decreasing in the first argument, it holds
\[
\alpha_\phi^{-1}([t,\infty))=\left[G(t,\phi),\infty\right)\quad\mbox{for all }t\in\mathbb{R}\mbox{ and }\phi\in\Phi,
\]
where $\alpha_\phi^{-1}([t,\infty)):=\{s\in\mathbb{R}\colon \alpha_{\min}^+(s,\phi)\in[t,\infty)\}$ is the preimage of $[t,\infty)$ under $\alpha_{\min}^+$. Fix $x\in X$. 
Suppose first that $\psi(x)<\infty$. 
It follows from the representation \eqref{eq:mainRepII} that $s\in[\psi(x),\infty)$
if and only if $\sup_{\phi\in\Phi} \{\phi(x)-\alpha_{\min}^+(s,\phi)\}\le 0$, if and only if $\phi(x)\le \alpha_{\min}^+(s,\phi)$ for all $\phi\in\Phi$, if and only if $s\in \alpha_\phi^{-1}([\phi(x),\infty))$ for all $\phi\in\Phi$. Hence,
$$
[\psi(x),\infty)=\bigcap_{\phi\in\Phi}\alpha_\phi^{-1}([\phi(x),\infty))=\bigcap_{\phi\in\Phi}\big[G(\phi(x),\phi),\infty\big).
$$
This shows $\psi(x)=\sup_{\phi\in\Phi} G(\phi(x),\phi)$ for all $x\in X$.
If $\psi(x)=\infty$, then the representation \eqref{eq:mainRepII} implies that for every $n\in\N$, there is $\phi_n\in \Phi$ such that $n\notin \alpha_\phi^{-1}([\phi_n(x),\infty))$.  
Hence, for every $n\in \N$, 
$$
n\notin \bigcap_{\phi\in\Phi}\alpha_\phi^{-1}([\phi(x),\infty))=\bigcap_{\phi\in\Phi}\big[G(\phi(x),\phi),\infty\big).
$$
This implies that $\psi(x)=\sup_{\phi\in\Phi} G(\phi(x),\phi)=\infty$.  

$(v)\Rightarrow(i)$ Since $\alpha^+_{\min}$ is countably stable in the second argument, we argue as in the implication from (iii) to (iv) that the level sets of $\psi$ are countably upwrads directed. Since $G$ is non-decreasing in the first argument, the function $\psi$ is monotone. 
\end{proof}

\begin{remark}
A modification of the previous result also leads to the representation of bounded completely maxitive functions\footnote{Defined similarly to Definition~\ref{def:countmax}, but for arbitrary subsets rather than countable subsets.}.  Indeed, Theorem~\ref{thm:main} remains valid if we replace countably upwards directed sublevel sets with completely upwards directed sublevel sets\footnote{I.e., for every \(s \in \mathbb{R}\), there exists \(y^\ast \in A_s\) with \(y \preccurlyeq y^\ast\) for all \(y \in A_s\).} and countable stability with complete stability\footnote{A function $\alpha\colon {\Phi}\to [-\infty,\infty]$ is  \emph{completely stable} if for every subset $Y\subset X$ with  $\sup_{y\in Y}\phi(y)\le \alpha(\phi)$ for all $\phi\in\Phi$,
there exists $y^\ast\in X$ with $y\preccurlyeq y^\ast$ for all $y\in Y$ and $\phi(y^\ast)\le \alpha(\phi)$ for all $\phi\in\Phi$.}.
The proof follows with the same arguments, subject to the following modifications. In the implication $(i)\Rightarrow(ii)$, we do not rely on Corollary~\ref{cor:repMax2}, but rather extend the definition \eqref{eq:defF} of $F(s,f)$ for all $s\in I$ and $f\in [-\infty,\infty)^\Phi$, where 
  $[-\infty,\infty)^\Phi$ denotes the set of all functions from $\Phi$ to $[-\infty,\infty)$. Direct verification shows that $F$ satisfies the conditions (1)--(3) on $I\times [-\infty,\infty)^\Phi$ and $F(s,\sup_{j\in J}f_j)=\sup_{j\in J} F(s,f_j)$ for every family $(f_j)_{j\in J}$ in $[-\infty,\infty)^\Phi$. 
  From the latter property, one obtains the representation~\eqref{eq:repMax} for all $s\in I$ and $f\in [-\infty,\infty)^\Phi$ for a suitable function $\alpha\colon I\times \Phi\to (-\infty,\infty]$. 
  From there, the argumentation is similar as in the proof of $(i)\Rightarrow(ii)$ in Theorem~\ref{thm:main}. 
\end{remark}

\begin{remark}
In line with \eqref{eq:rep pw}, \eqref{eq:rep st}, and \eqref{eq:rep general}, the function \(\alpha\) in \eqref{eq:mainRep} can be interpreted as a penalty function. 
In the present setting, there may exist different penalties \(\alpha\) for which the representation \eqref{eq:mainRep} holds. 

We note that \(\alpha_{\min}\) is the smallest penalty \(\alpha\) for which the representation \eqref{eq:mainRep} is valid. Indeed, if \eqref{eq:mainRep} holds for a penalty \(\alpha\), then for any fixed \(\phi \in \Phi\) and \(s \in \mathbb{R}\), it follows that \(\phi(x) \leq \alpha(s, \phi)\) for all \(x \in A_s\). Taking the supremum over all \(x \in A_s\), we obtain that \(\alpha_{\min}(s, \phi) = \sup_{x \in A_s} \phi(x) \leq \alpha(s, \phi)\). Furthermore, by construction, \(\alpha_{\min}^+\) is the smallest penalty for which the representation \eqref{eq:mainRep} holds and is also right-continuous in the first argument. Note that the right continuity of \(\alpha_{\min}^+\) is necessary in the implication \((iv)\Rightarrow(v)\) in Theorem \ref{thm:main}.
\end{remark}

For the remainder of this section, \( X \) is endowed with an operation \( + \colon X \times \mathbb{R} \to X \). A function \( \psi \colon X \to [-\infty, \infty] \) satisfies the \emph{translation property} if \( \psi(x + c) = \psi(x) + c \) for all \( x \in X \) and \( c \in \mathbb{R} \). A function \( \psi \colon X \to [-\infty, \infty] \) satisfies the \emph{deviation property} if \( \psi(x + c) = \psi(x) \) for all \( x \in X \) and \( c \in \mathbb{R} \).

\begin{corollary}\label{cor:translationproperty}
Suppose that every \( \phi \in \Phi \) has the translation property.
Let \( \psi \colon X \to (-\infty, \infty] \) be a function with \( A_s := \{ x \in X \colon \psi(x) \le s \} \) for all \( s \in \mathbb{R} \). Then, the following statements are equivalent:
\begin{itemize}
\item[(i)] \( \psi \) has the translation property, is monotone, and has countably upwards directed sublevel sets.
\item[(ii)] There exists a countably stable function \( \gamma \colon \Phi \to [-\infty, \infty] \) with
\[
\psi(x) = \sup_{\phi \in \Phi} \{\phi(x) - \gamma(\phi)\} \quad \text{for all } x \in X.
\]
\item[(iii)] The mapping \( \phi \mapsto \alpha_{\min}(0, \phi) := \sup_{x \in A_0} \phi(x) \) is countably stable and
\[
\psi(x) = \sup_{\phi \in \Phi} \{\phi(x) - \alpha_{\min}(0, \phi)\} \quad \text{for all } x \in X.
\]
\end{itemize}
\end{corollary}
\begin{proof}
$(i)\Rightarrow(iii)$: Since $\psi$ has the translation property, it holds $\psi(x)\le s$ if and only if $\psi(x-s)\le 0$, so that 
$A_s=A_0-s:=\{x-s\colon x\in A_0\}$ for all $s\in \mathbb{R}$.
Hence, since every $\phi\in\Phi$ has the translation property, we obtain
for every $s\in\mathbb{R}$ and $\phi\in\Phi$,
\[
\alpha_{\min}(s,\phi)=\sup_{x\in A_s}\phi(x)=\sup_{x\in A_0}\phi(x+s)=\sup_{x\in A_0}\phi(x)+s=\alpha_{\min}(0,\phi)+s.
\]
It follows from Theorem~\ref{thm:main} that for every $x\in X$,
 \begin{align*}
  \psi(x)&=\inf\left\{s \in\mathbb{R}\colon \underset{\phi\in\Phi}\sup\{\phi(x)-\alpha_{\min}(s,\phi)\} \le 0\right\}\\
  &=\inf\left\{s \in\mathbb{R}\colon \underset{\phi\in\Phi}\sup\{\phi(x)-\alpha_{\min}(0,\phi)\} \le s\right\}\\
  &=\underset{\phi\in\Phi}\sup\{\phi(x)-\alpha_{\min}(0,\phi)\}.
 \end{align*}
 
 $(iii)\Rightarrow(ii)$ is obvious.
 
 $(ii)\Rightarrow(i)$: That $\psi$ is monotone and has countably upwards directed sublevel sets follows from Theorem~\ref{thm:main}. Moreover, since every $\phi\in\Phi$ has the translation property, for every $x\in X$ and $c\in\mathbb{R}$,
 \[
\psi(x+c)=\sup_{\phi\in\Phi}\{\phi(x+c)-\gamma(\phi)\}=\sup_{\phi\in\Phi}\{\phi(x)+c-\gamma(\phi)\}=\psi(x)+c,
 \]
which shows that $\psi$ has the translation property.
\end{proof}

\begin{corollary}
Suppose that every \( \phi \in \Phi \) satisfies the deviation property, and let \( \psi \colon X \to (-\infty, \infty] \) be a monotone function with countably upwards directed sublevel sets. If \( \psi \) additionally satisfies the translation property, then \( \psi(x) = \infty \) for all \( x \in X \).
\end{corollary}
\begin{proof}
It follows from Theorem \ref{thm:main} that for every $x\in X$,
\[
\psi(x)=\inf\left\{s\in\mathbb{R}\colon \phi(x)\le \alpha_{\min}(s,\phi)\mbox{ for all }\phi\in\Phi\right\}.
\]
By similar arguments as in Corollary~\ref{cor:translationproperty}, it holds $\alpha_{\min}(s,\phi)=\alpha_{\min}(0,\phi)$ for all $s\in\mathbb{R}$. Hence, it necessary holds
$\psi(x)=\infty$. 
\end{proof}

\section{Maxitivity with respect to stochastic orders}\label{sec:ex}

Let $(\Omega,\mathcal{F},\mathbb{P})$ be an atomless probability space.\footnote{Recall that a probability space $(\Omega,\mathcal{F},\mathbb{P})$ is atomless if for all $A\in\mathcal{F}$ with $\mathbb{P}(A)>0$ there exists $B\in\mathcal{F}$ such that $B\subset A$ and $0<\mathbb{P}(B)<\mathbb{P}(A)$; see e.g.~\cite{delbaen0}.}  
We denote by $\mathcal{M}$ the set of all random variables $\xi\colon \Omega\to\R$, where two of them are identified if they have the same law.
Further, let  $\mathcal{M}_\infty$ and $\mathcal{M}_1$ be the sets of all random variables in $\mathcal{M}$ which are essentially bounded, and have finite first moments, respectively. 
In the following, we analyze the main results of the previous section for several stochastic orders. 
For definitions and further details on stochastic orders, we refer to~\cite{muller2002,shaked2007}.

\subsection{Usual stochastic order}
For $\xi\in \mathcal{M}$, the \emph{quantile function} $q_\xi\colon (0,1)\to\mathbb{R}$ is defined as the left-inverse of the cumulative distribution function, i.e.,
\[
q_\xi(u):=
\inf\{v\in\mathbb{R}\colon u\le \mathbb{P}(\xi\le v)\}.
\]
On the one hand, the quantile function $q_\xi$ is non-decreasing and left-continuous, see Lemma~\ref{lem:basicQuantile}.  
On the other hand, every function $q\colon (0,1)\to\mathbb{R}$ which is non-decreasing and left-continuous satisfies $q=q_\xi$ for some $\xi\in\mathcal{M}$, see Lemma~\ref{lem:basicQuantile2}. 
Hence, there is a one-to-one relation between the set $\mathcal{M}$ and quantile functions.

The \emph{usual stochastic order} on $\mathcal{M}$ is defined by  
\[
\xi\preccurlyeq_{\rm st} \eta\quad\mbox{if and only if}\quad q_\xi(u)\le q_\eta(u)\mbox{ for all }u\in(0,1).  
\]
 The ordered space $(\mathcal{M},\preccurlyeq_{\rm st})$ is a  complete lattice, see~\cite{kertz,muller}. 
In particular, every  sequence $(\xi_n)_{n\in\mathbb{N}}\subset \mathcal{M}$ which is bounded from above has a supremum  $\xi^\ast=\vee_{n\in\N}\xi_n$ in $\mathcal{M}$ which is determined by   
 $q_{\xi^\ast}(u)=\sup_{n\in\N}q_{\xi_n}(u)$ for all $u\in(0,1)$.
  
Let $\Phi_{\rm st}:=\{\phi_u\colon u\in(0,1)\}$ denote the family of functions
$\phi_u\colon \mathcal{M}\to \mathbb{R}$ given by $\phi_u(\xi):=q_{\xi}(u)$. We identify $\Phi_{\rm st}$ with $(0,1)$, and endow $(0,1)$ with the Euclidean topology.\footnote{$(0,1)$ is not complete with the Euclidean distance. However, it is completely metrizable, that is, there exists an equivalent metric on $(0,1)$ such that $(0,1)$ with the relative Euclidean topology is a complete metric space space. 
In fact, every countable intersection of open subsets of a complete metric space is completely metrizable, see~\cite[Theorem 24.12]{willard}.} 
Then, it holds 
$\xi \preccurlyeq_{\rm st} \eta$ if and only if $\phi_u(\xi)\le \phi_u(\eta)$ for all $u\in(0,1)$. Moreover, for 
 every $\xi\in\mathcal{M}$, the mapping $u\mapsto \phi_u(\xi)$ is lower semicontinuous. Therefore, the partially ordered space $(\mathcal{M},\preccurlyeq_{\rm st})$ satisfies the assumptions of Section \ref{sec:mainResult}. 
 \begin{lemma}\label{lem:stfunctions}
  Every function $\gamma\colon (0,1)\to [-\infty,\infty)$ is countably stable.    
 \end{lemma}
 \begin{proof}
Let $(\xi_n)_{n\in\mathbb{N}}\subset \mathcal{M}$ with $q^\ast(u):=\sup_{n\in\N} q_{\xi_n}(u)\le \gamma(u)$ for all $u\in(0,1)$. Then, the mapping $u\mapsto q^\ast(u)$ is  real-valued, non-decreasing and left-continuous. Hence, there exists $\xi^\ast\in\mathcal{M}$ with $q^\ast(u)=q_{\xi^\ast}(u)$ for all $u\in(0,1)$. 
In particular, $\xi^\ast=\vee_{n\in\N}\xi_n$ and $\phi_u(\xi^\ast)=q^\ast(u)\le \gamma(u)$ for all $u\in (0,1)$.
\end{proof}
For a function $\psi\colon \mathcal{M}\to (-\infty,\infty]$ 
with sublevel sets $A_s:=\{\xi\in \mathcal{M}\colon \psi(\xi)\le s\}$, we define $\alpha_{\min}(s,u):=\sup_{\xi\in A_s}q_\xi(u)$ and $\alpha^+_{\min}(s,u):=\inf_{s<t}\alpha_{\min}(t,u)$
for all $s\in\R$ and $u\in(0,1)$. If $\psi$ is countably maxitive, then the sublevel sets are countably upwards directed, so that $u\mapsto\alpha_{\min}(s,u)$ is real-valued for all $s\in\R$ with $A_s\neq \emptyset$. 
As a direct application of Theorem~\ref{thm:main} and Lemma~\ref{relation:maxitive}, we obtain the following representation result for maxitive functions w.r.t.~the usual stochastic order.
 A related representation result for maxitive functionals with respect to the usual stochastic order was investigated in parallel to our work in \cite[Theorem 1]{wang2}.

\begin{corollary}
Let $\psi\colon \mathcal{M}\to (-\infty,\infty]$ be a function. 
Then, the following conditions are equivalent:
\begin{itemize}
    \item[(i)] $\psi$ is bounded countably maxitive w.r.t. $\preccurlyeq _{\rm st}$.
    \item[(ii)] $\psi$ is monotone and has countably upwards directed sublevel sets w.r.t. $\preccurlyeq _{\rm st}$.
    \item[(iii)] There exists a function $\alpha\colon \mathbb R\times (0,1)\to [-\infty,\infty)$ which is non-decreasing in the first argument such that 
    \[
    \psi(\xi)=\inf\left\{ s\in\mathbb{R} \colon \sup_{u\in(0,1)}\{q_\xi(u)-\alpha(s,u)\}\le 0\right\}\quad\mbox{for all }\xi\in\mathcal{M}.
    \]
    The infimum is attained for all $\xi\in\mathcal{M}$ with $\psi(\xi)<\infty$.  
\item[(iv)] The mapping $u\mapsto\alpha_{\min}(s,u)$ is real-valued for all $s\in\R$ with $A_s\neq \emptyset$. Moreover, it holds
\[
\psi(\xi)=\sup_{u\in(0,1)}G(q_\xi(u),u)\quad\mbox{ for all }\xi\in \mathcal{M},
\]
where $G(t,u):=\inf\{s\in\mathbb{R}\colon t\le \alpha_{\min}^+(s,u)\}$ for all $(t,u)\in \mathbb{R}\times (0,1)$.
\end{itemize}
\end{corollary}

If $\psi$ also has the translation property, then 
we get the following version of the previous result.
We notice that $q_{\xi+c}=q_{\xi}+c$
for all $\xi\in\mathcal{M}$ and $c\in\R$, i.e., the mapping $\phi_u\colon\mathcal{M}\to\R$ has the translation property for all $u\in(0,1)$. Hence, we can apply Corollary~\ref{cor:translationproperty} and obtain the following result. 
\begin{corollary}\label{cor:st}
Let $\psi\colon \mathcal{M}\to (-\infty,\infty]$ be a function with the translation property.
Then, the following conditions are equivalent:
\begin{itemize}
    \item[(i)] $\psi$ is bounded countably maxitive  w.r.t. $\preccurlyeq _{\rm st}$. 
    \item[(ii)] There exists a function $\alpha\colon \R\times(0,1)\to [-\infty,\infty)$ such that
\[
\psi(\xi)=\sup_{u\in(0,1)}\{q_\xi(u)-\alpha(u)\}\quad\mbox{for all }\xi\in \mathcal{M}.
\]
\item[(iii)] The mapping $u\mapsto \alpha_{\min}(0,u)=\sup_{\xi\in A_0}q_\xi(u)$ is real-valued, and 
    \[
\psi(\xi)=\sup_{u\in(0,1)}\{q_\xi(u)-\alpha_{\min}(0,u)\} \quad\mbox{ for all }\xi\in \mathcal{M}.
\]
\end{itemize}
\end{corollary}

Recall that the \emph{value at risk} ${\rm VaR}_u\colon \mathcal{M}\to \mathbb{R}$ at level $u\in(0,1)$, 
which is defined by ${\rm VaR}_u(\xi):=q_\xi(u)$,
is a monetary risk measure; see \cite{follmer}.
The previous Corollary~\ref{cor:st} is closely related to 
\cite[Theorem 4]{wang} which provides a representation result 
for lower semicontinuous completely maxitive functions
$\mathcal{M}_1\to\R$ having the translation property. A related representation result was also shown in \cite[Proposition 4]{bignozzi} for the class of benchmark-loss VaR, which were introduced by Bignozzi et al.~\cite{bignozzi}. 
\begin{remark}
For every sequence $(\xi_n)_{n\in\N}\subset\mathcal{M}$ such that $\vee_{n\in\N}\xi_n$ exists, it holds ${\rm VaR}_u(\vee_{n\in\N} \xi_n)=\sup_{n\in\N}{\rm VaR}_u(\xi_n)$.
Nevertheless, the function ${\rm VaR}_u\colon \mathcal{M}\to \mathbb{R}$ is not countably bounded maxitive because it may happen that $\sup_{n\in\N}{\rm VaR}_u(\xi_n)<\infty$ for an unbounded sequence $(\xi_n)_{n\in\N}\subset\mathcal{M}$. 
However, ${\rm VaR}_u$ can also be 
defined on the set $\mathcal{M}^\ast$ of all extended random variables $\xi\colon \Omega\to(-\infty,\infty]$. 
Then, the supremum $\vee_{i\in I} \xi_i$ exists for every family $(\xi_i)_{i\in I}\subset \mathcal{M}^\ast$ and ${\rm VaR}_u\colon \mathcal{M}^\ast\to (-\infty,\infty]$ is completely maxitive. 
\end{remark}

\subsection{Increasing convex  stochastic order} 
For $\xi\in \mathcal{M}_1$, we define the 
\emph{integrated quantile function} $Q_\xi\colon (0,1)\to\mathbb{R}$  by
\[
Q_\xi(u):=\int_u^1 q_\xi(v)\,{\rm d}v.
\]
The function $u\mapsto Q_\xi(u)$ is concave, with $Q_\xi(0^+)=\mathbb{E}[\xi]$ and $Q_\xi(1^-)=0$. 
Conversely, every concave function $Q\colon (0,1)\to \mathbb{R}$ with $Q(0^+)\in\mathbb{R}$ and $Q(1^-)=0$ is of the form  $Q=Q_\xi$ for some random variable $\xi\in \mathcal{M}_1$, see~Lemma~\ref{lem:basicQuantile3}.  
The \emph{increasing convex order} on $\mathcal{M}_1$ is defined by 
\[
\xi\preccurlyeq _{\rm icx} \eta\quad\mbox{if and only if}\quad Q_\xi(u)\le Q_\eta(u)\quad\mbox{ for all }u\in(0,1).  
\]
The ordered space $(\mathcal{M}_1,\preccurlyeq_{\rm icx})$ is a  complete lattice, see~\cite{kertz,muller}. In particular,  
 every  sequence $(\xi_n)_{n\in\mathbb{N}}\subset \mathcal{M}_1$ which is bounded from above has a supremum  $\xi^\ast=\vee_{n\in\N}\xi_n$ which is determined by   
 \[
  Q_{\xi^\ast}(u)={\rm conc}(Q_{\xi_1},Q_{\xi_2},\cdots)(u)\quad\mbox{ for all }u\in(0,1).
 \]  
Here, ${\rm conc}(Q_{\xi_1},Q_{\xi_2},\cdots)$ denotes the \emph{concave hull} of the sequence $(Q_{\xi_n})_{n\in\mathbb{N}}$, i.e.,
the smallest concave function $Q^\ast\colon (0,1)\to [-\infty,\infty]$ with $\sup_{n\in\mathbb{N}}Q_{\xi_n}(u)\le Q^\ast(u)$ for all $u\in(0,1)$.\footnote{Notice that ${\rm conc}(Q_{\xi_1},Q_{\xi_2},\cdots)$ is the pointwise infimum over all concave functions $Q\colon(0,1)\to(-\infty,\infty]$ with  
$Q_{\xi_n}(u)\le Q(u)$ for all $u\in(0,1)$ and $n\in\mathbb{N}$.} 

Let $\Phi_{\rm icx}:=\{\phi_u\colon u\in (0,1)\}$ be the family 
of functions $\phi_u\colon\mathcal{M}_1\to\R$ given by $\phi_u(\xi):=Q_\xi(u)$. We identify $\Phi_{\rm icx}$ with $(0,1)$, and endow $(0,1)$ with the Euclidean topology. For every $\xi\in\mathcal{M}_1$, the function $u\mapsto \phi_u(\xi)$ is continuous. Hence, the partially ordered space $(\mathcal{M}_1,\preccurlyeq_{\rm icx})$ satisfies the assumptions of Section \ref{sec:mainResult}. 
\begin{lemma}\label{lem:icxfunctions}
Every concave function $\gamma\colon (0,1)\to [-\infty,\infty)$ is countably stable. 
\end{lemma}
 \begin{proof}
Let $(\xi_n)_{n\in\mathbb{N}}\subset \mathcal{M}_1$ with $\vee_{n\in\N} Q_{\xi_n}(u)\le \gamma(u)$ for all $u\in(0,1)$. 
Since $\gamma$ is concave, it holds $Q^\ast(u):={\rm conc}(Q_{\xi_1},Q_{\xi_2},\cdots)(u)\le \gamma(u)$ for all $u\in(0,1)$. 
Since $u\mapsto Q^\ast(u)$ is real-valued, we obtain $Q_{\xi^\ast}=Q^\ast$, where $\xi^\ast=\vee_{n\in\N}\xi_n\in\mathcal{M}_1$, and $Q_{\xi^\ast}(u)=Q^\ast(u)\le \gamma(u)$ for all $u\in(0,1)$. 
 \end{proof}

For a function $\psi\colon \mathcal{M}_1\to(-\infty,\infty]$
with sublevel sets $A_s:=\{\xi\in \mathcal{M}_1\colon \psi(\xi)\le s\}$, we define 
$\alpha_{\min}(s,u):=\sup_{\xi\in A_s}Q_\xi(u)$ and $\alpha^+_{\min}(s,u):=\inf_{s<t}\alpha_{\min}(t,u)$ for all $s\in\R$ and $u\in(0,1)$.
If $\psi$ is countably maxitive, then the sublevel sets are countably upwards directed, from which it is straightforward to verify that
the mapping $u\mapsto\alpha_{\min}(s,u)$ is real-valued and concave for all $s\in\R$ with $A_s\neq \emptyset$.\footnote{As for the concavity, we notice that for $\xi_1,\xi_2\in A_s$, there exists $\xi^\ast\in A_s$ with $Q_{\xi^\ast}={\rm conc}(Q_{\xi_1},Q_{\xi_2})$ such that $\lambda Q_{\xi_1}(u_1)+(1-\lambda) Q_{\xi_2}(u_2)\le Q_{\xi^\ast}(\lambda u_1+(1-\lambda)u_2)$ for all $u_1,u_2,\lambda\in(0,1)$.} As a direct application of Theorem~\ref{thm:main} and Lemma~\ref{relation:maxitive}, we obtain the following representation result for maxitive functions w.r.t.~the increasing convex stochastic order.

\begin{corollary}\label{cor:icx}
Let $\psi\colon  \mathcal{M}_1\to (-\infty,\infty]$ be a function. 
Then, the following conditions are equivalent:
\begin{itemize}
    \item[(i)] $\psi$ is bounded countably maxitive w.r.t.~$\preccurlyeq _{\rm icx}$.
     \item[(ii)] $\psi$ is monotone and has countably upwards directed sublevel sets w.r.t.~$\preccurlyeq _{\rm icx}$.
    \item[(iii)] There exists a function $\alpha\colon \mathbb R\times (0,1)\to [-\infty,\infty)$ which is non-decreasing in the first argument and concave in the second argument such that 
    \[
    \psi(\xi)=\inf\left\{ s\in\mathbb{R} \colon \sup_{u\in(0,1)}\{Q_\xi(u)-\alpha(s,u)\}\le 0\right\}\quad\mbox{for all }\xi\in\mathcal{M}_1. 
    \]
    The infimum is attained for all $\xi\in\mathcal{M}_1$ with $\psi(\xi)<\infty$. 
\item[(iv)] The mapping $u\mapsto\alpha_{\min}(s,u)$ is real-valued and concave for all $s\in\R$ with $A_s\neq \emptyset$. Moreover, it holds
 \[
\psi(\xi)=\sup_{u\in(0,1)}G(Q_\xi(u),u)\quad\mbox{ for all }\xi\in\mathcal{M}_1,
\]
where $G(t,u):=\inf\{s\in\mathbb{R}\colon t\le \alpha_{\min}^+(s,u)\}$ for all $(t,u)\in \mathbb{R}\times (0,1)$.
\end{itemize}
\end{corollary}

If $\psi$ in addition has the translation property, then we can argue as follows. Since the function $\xi\mapsto Q_\xi$ does not have the translation property, we consider the modified family $\Phi^\prime_{\rm icx}:=\{{\rm ES}_u\colon u\in (0,1)\}$. Here, ${\rm ES}_u\colon \mathcal{M}_1\to \mathbb{R}$ is the \emph{expected shortfall}  given by
 \[
 {\rm ES}_u(\xi)
 := \frac{1}{1-u}Q_{\xi}(u)\quad\mbox{ for all }u\in(0,1).
 \]
 The expected shortfall is a convex risk measure; see~\cite{follmer}. 
 In particular, the functions ${\rm ES}_u$ have the translation property
 and also induce the partial order $\preccurlyeq _{\rm icx}$.  
 However, by changing to the family $\Phi^\prime_{\rm icx}$, we 
 also change the set of countably stable functions. 
 More precisely, a function $\gamma\colon(0,1)\to [-\infty,\infty]$ is countably stable w.r.t.~$\Phi_{\rm icx}$ if and only if the function $u\mapsto \frac{1}{1-u}\gamma(u)$ is countably stable w.r.t.~$\Phi_{\rm icx}^\prime$. Moreover, 
 the minimal penalty function $\alpha_{\min}^\prime$ w.r.t. $\Phi_{\rm icx}^\prime$ is given by $$\alpha_{\min}^\prime(s,u)=\tfrac{1}{1-u}\alpha_{\min}(s,u).$$ 
As a consequence of Corollary~\ref{cor:translationproperty}, we obtain the following result. 
 \begin{corollary}\label{cor:shortFall}
Let $\psi\colon\mathcal{M}_1\to(-\infty,\infty]$ be a function with the translation property. Then, the following conditions are equivalent:
\begin{itemize}
\item[(i)] $\psi$ is bounded countably maxitive w.r.t. $\preccurlyeq_{\rm icx}$.
\item[(ii)] There exists a concave function $\alpha\colon(0,1)\to [-\infty,\infty)$ such that 
\[
\psi(\xi)=\sup_{u\in (0,1)}\left\{{\rm ES}_u(\xi)-\tfrac{1}{1-u}\alpha(u)\right\}\quad\mbox{for all }\xi\in\mathcal{M}_1.
\]

\item[(iii)] The mapping $u\mapsto \alpha_{\min}(0,u)$ is real-valued and concave. Moreover, it holds
\[
\psi(\xi)=
\sup_{u\in (0,1)}\left\{{\rm ES}_u(\xi)-\tfrac{1}{1-u}\alpha_{\min}(0,u)\right\}\quad\mbox{for all }\xi\in\mathcal{M}_1.
\]
\end{itemize}
\end{corollary} 
The previous Corollary~\ref{cor:shortFall} is closely related to
\cite[Theorem 5]{wang}. Moreover, it is shown in \cite{wang} that the function ${\rm ES}_u$ does not always preserve the maximum of two randoms variables w.r.t.~$\preccurlyeq_{\rm icx}$, i.e., the function ${\rm ES}_u$ is not maxitive w.r.t.~$\preccurlyeq_{\rm icx}$. 

\subsection{Convex stochastic order}
The \emph{convex order} on $\mathcal{M}_1$ is defined by  $$\xi\preccurlyeq _{\rm cx} \eta\quad\mbox{if and only if}\quad
\mathbb{E}[\xi]=\mathbb{E}[\eta]\mbox{ and }\xi\preccurlyeq _{\rm icx} \eta.$$ 
The partially ordered space $(\mathcal{M}_1,\preccurlyeq _{\rm cx})$ is not a lattice on $\mathcal{M}_1$ because it is not possible to define the maximum of two random variables $\xi,\nu\in \mathcal{M}_1$ for which $\mathbb{E}[\xi]\neq \mathbb{E}[\eta]$.  
However, for $r\in\R$, the convex order agrees with the increasing convex order on the set
\[
{\mathcal{M}}_1^r:=\left\{ \xi\in{\mathcal{M}}_1\colon \mathbb{E}[\xi]=r\right\}.
\]
In particular, $(\mathcal{M}^r_1,\preccurlyeq _{\rm cx})$ is a complete lattice. For $\psi\colon \mathcal{M}^r_1\to(-\infty,\infty]$
with sublevel sets $A_s:=\{\xi\in \mathcal{M}^r_1\colon \psi(\xi)\le s\}$, we define 
$\alpha_{\min}\colon \mathbb{R}\times (0,1)\to [-\infty,\infty]$ by $\alpha_{\min}(s,u):=\sup_{\xi\in A_s}Q_\xi(u)$ and $\alpha^+_{\min}(s,u):=\inf_{s<t}\alpha_{\min}(t,u)$. 
As a direct application of Theorem~\ref{thm:main} and Lemma~\ref{relation:maxitive}, we conclude that a function \(\psi\colon \mathcal{M}_1^r \to \mathbb{R}\) is countably maxitive with respect to \(\preccurlyeq_{\rm cx}\), if it admits a representation similar to that in Corollary~\ref{cor:icx}, with $\mathcal{M}_1$ replaced by $\mathcal{M}_1^r$.

\subsection{Increasing concave  stochastic order}
For $\xi\in\mathcal{M}_1$, we define the function $\bar Q_\xi\colon (0,1)\to\mathbb R$ by
\[
\bar Q_\xi(u):=-{Q}_{-\xi}(u)=-\int_u^1 q_{-\xi}(v)\,{\rm d} v.
\]
The mapping $t\mapsto \bar Q_\xi(t)$ is convex and satisfies $\bar Q_\xi(0^+)=\mathbb{E}[\xi]$ and $\bar Q_\xi(1^-)=0$. Conversely, every convex function $\bar Q\colon (0,1)\to \mathbb{R}$ 
with $\bar Q(0^+)\in\mathbb{R}$ and $\bar Q(1^-)=0$ is of the form $\bar Q=\bar Q_\xi$ for some random variable $\xi\in \mathcal{M}_1$. 
The \emph{increasing concave stochastic order}  is defined by  $$\xi \preccurlyeq_{\rm icv} \eta\quad\mbox{if and only if }\quad -\eta\preccurlyeq_{\rm icx} -\xi.$$ That is,  $\xi \preccurlyeq_{\rm icv} \eta$ if and only if $\overline{Q}_\xi(u)\le \overline{Q}_\eta(u)$ for all $u\in(0,1)$. 
The partially ordered space $(\mathcal{M}_1,\preccurlyeq_{\rm icv})$ is a  complete lattice, see~\cite{kertz,muller}. 
In particular, every  sequence $(\xi_n)_{n\in\N}\subset \mathcal{M}_1$ which is bounded from above w.r.t. $\preccurlyeq_{\rm icv}$ has a supremum  $\xi^\ast=\vee_{n\in\N}\xi_n$ which is determined by $\bar Q_{\xi^\ast}(u)=\vee_{n\in \N} \bar Q_{\xi_n}(u)$ for all $u\in(0,1)$. We consider the family $\Phi_{\rm icv}:=\{\phi_{u}\colon u\in(0,1)\}$ of functions $\phi_u\colon\mathcal{M}_1\to\mathbb{R}$ given by $\phi_u(\xi):=\bar{Q}_\xi(u)$. 
 Hence, the partially ordered space $(\mathcal{M}_1,\preccurlyeq_{\rm icv})$ satisfies the conditions of Section \ref{sec:mainResult}. 
With arguments similar to those for the usual stochastic order, we get the following result.
\begin{lemma}\label{lem:icvfunctions}
Every function $\gamma\colon (0,1)\to [-\infty,\infty)$  is countably stable.    
 \end{lemma}
 
For a function $\psi\colon \mathcal{M}\to (-\infty,\infty]$ 
with sublevel sets $A_s:=\{\xi\in \mathcal{M}_1\colon \psi(\xi)\le s\}$, we define 
 $\alpha_{\min}(s,u):=\sup_{\xi\in A_s}\bar Q_\xi(u)$ and $\alpha^+_{\min}(s,u):=\inf_{s<t}\alpha_{\min}(t,u)$ for all 
 $s\in\R$ and $u\in(0,1)$. If $\psi$ is countably maxitive, then the sublevel sets are countably upwards directed, so that $u\mapsto\alpha_{\min}(s,u)$ is real-valued for all $s\in\R$ with $A_s\neq \emptyset$.  
As a direct application of Theorem~\ref{thm:main} and Lemma~\ref{relation:maxitive}, we obtain the following representation result for maxitive functions w.r.t.~the increasing concave stochastic order. 
\begin{corollary}
Let $\psi\colon \mathcal{M}_1\to (-\infty,\infty]$ be a function. 
Then, the following conditions are equivalent.
\begin{itemize}
    \item[(i)] $\psi$ is bounded countably maxitive w.r.t. $\preccurlyeq _{\rm icv}$.
    \item[(ii)] $\psi$ is monotone and has countably upwards directed sublevel sets w.r.t. $\preccurlyeq _{\rm icv}$.
    \item[(ii)] There exists a function $\alpha\colon \mathbb R\times (0,1)\to [-\infty,\infty)$ which is non-decreasing in the first argument such that 
    \[
    \psi(\xi)=\inf\left\{ s\in\mathbb{R} \colon \sup_{u\in(0,1)}\{\bar Q_\xi(u)-\alpha(s,u)\}\le 0\right\}\quad\mbox{for all }\xi\in\mathcal{M}_1.
    \]
    The infimum is attained for all $\xi\in\mathcal{M}_1$ with $\psi(\xi)<\infty$.  
\item[(iii)] The mapping $u\mapsto\alpha_{\min}(s,u)$ is real-valued for all $s\in\R$ with $A_s\neq \emptyset$. Moreover, it holds
    \[
\psi(\xi)=\sup_{u\in(0,1)}G\left(\bar Q_\xi(u),u\right)\quad\mbox{ for all }\xi\in \mathcal{M},
\]
where $G(t,u):=\inf\{s\in\mathbb{R}\colon t\le \alpha_{\min}^+(s,u)\}$ for all $(t,u)\in \mathbb{R}\times (0,1)$.
\end{itemize}
\end{corollary}

In case that $\psi$ has the translation property, we can proceed as follows. Since the function $\xi\mapsto \bar Q_\xi$ does not have the translation property, we consider the modified family 
$\Phi_{icv}^\prime:=\{\overline{\rm ES}_u\colon u\in(0,1)\}$, where $\overline{{\rm ES}}_u\colon \mathcal{M}_1\to\R$ is defined by
$\overline{{\rm ES}}_u(\xi):=-{{\rm ES}}_u(-\xi)$.
 The functions $\overline{{\rm ES}}_u$ are translation invariant and induce the order $\preccurlyeq_{\rm icv}$. The families $\Phi_{\rm icv}^\prime$ and $\Phi_{\rm icv}$ result in the same countably stable functions. In addition, we obtain
 \[
 \alpha_{\min}^\prime(s,u)=\frac{1}{1-u}\alpha_{\min}(s,u),
 \]
 where $\alpha_{\min}^\prime$ is the minimal penalty function w.r.t. $\Phi_{\rm icv}^\prime$. 
As a consequence of Corollary~\ref{cor:translationproperty}, we conclude that a function  $\psi\colon\mathcal{M}_1\to(-\infty,\infty]$ with the translation property is countably maxitive with respect to  $\preccurlyeq_{\rm icv}$, if it admits a representation similar to that in Corollary~\ref{cor:shortFall}, with ${\rm ES}_u$ replaced by $\overline{{\rm ES}}_u$.

\subsection{Dispersive stochastic order}\label{sec:disp order}
The \emph{dispersive stochastic order} on $\mathcal{M}$ is defined by 
\[
\xi\preccurlyeq _{\rm disp} \eta\quad\mbox{if and only if}\quad
q_\xi(v)-q_\xi(u)\le q_\eta(v)-q_\eta(u)\mbox{ for all }0<u\le v < 1.  
\]
The dispersive order is not antisymmetric. However, by considering the quotient space $\mathcal{M}/_\sim$ w.r.t.~the equivalence relation
$\xi{\sim} \eta$ if and only if  $\xi=\eta+c$ for some $c\in\R$, we obtain that 
the partially ordered space $(\mathcal{M}/_\sim,\preccurlyeq_{\rm disp})$ is a lattice, see~\cite{muller}. 

\smallskip
For a non-empty set $Q$ of quantile functions and $[u,v]\subset [0,1]$, we define its \emph{total variation} on $[u,v]$ as
\[
{\rm TV}_{[u,v]}\left(Q\right)
:=\sup_{\pi\in \Pi[u,v]}S_\pi(Q), 
\]
where $\Pi[u,v]$ consists of all partitions $\pi=\{u=t_0<t_1<\cdots<t_n=v\}$ for some $n\in\N$. Here, we set $S_\pi(Q):=\sum_{i=1}^n  \sup_{q\in Q}(q(t_i)-q(t_{i-1}))$ with $q(0):=q(0^+)$ and $q(1):=q(1^-)$. We say that the set $Q$ is of \emph{bounded variation} if ${\rm TV}_{[0,1]}(Q)<\infty$. 
It is straightforward to verify that ${\rm TV}_{[0,v]}={\rm TV}_{[0,u]}+{\rm TV}_{[u,v]}$ for all $0<u\le v<1$.

\begin{proposition}\label{prop:existenceSup}
 Suppose that $Q:=\{q_\xi\colon \xi\in A\}$ is of bounded variation
 for some  $A\subset \mathcal{M}$.  
 Then, there exists $\xi^\ast\in \mathcal{M}$ such that $q_{\xi^\ast}(t)={\rm TV}_{[0,t]}(Q)$ for all $t\in (0,1]$. 
 Moreover, $\xi^\ast$ is a supremum of $A$ w.r.t. $\preccurlyeq _{\rm disp}$.
\end{proposition}
\begin{proof}
By Proposition~\ref{prop:leftContTV} and the fact that $Q$ is of bounded variation, the function  $t\mapsto {\rm TV}_{[0,t]}(Q)$ is real-valued, non-decreasing and left-continuous. Hence, there exists $\xi^\ast\in \mathcal{M}$ such that $q_{\xi^\ast}(t)={\rm TV}_{[0,t]}(Q)$ for all $t\in (0,1]$. 
 Given $u,v\in\R$ with $0<u\le v<1$, for every $\xi\in A$, we have that
 \[
 q_\xi(v)-q_\xi(u)\le {\rm TV}_{[u,v]}(Q)=  {\rm TV}_{[0,v]}(Q) - {\rm TV}_{[0,u]}(Q)  =q_{\xi^\ast}(v)-q_{\xi^\ast}(u).
 \]
 Thus, $\xi^\ast$ is an upper bound of $A$ w.r.t. $\preccurlyeq _{\rm disp}$. 
 On the other hand, if $\eta$ is another upper bound of $A$  w.r.t.~$\preccurlyeq _{\rm disp}$, we necessarily have that
 \[
 q_{\xi^\ast}(v)-q_{\xi^\ast}(u)={\rm TV}_{[u,v]}(Q)\le {\rm TV}_{[u,v]}(q_\eta)=q_{\eta}(v)-q_{\eta}(u),
 \]
 where we have used in the last equality that $q_{\eta}$ is non-decreasing. 
 This implies that $\xi^\ast \preccurlyeq _{\rm disp}\eta$.
\end{proof}

\begin{corollary}\label{lem:disp_complete}
 The partially ordered space $(\mathcal{M}_\infty/_\sim,\preccurlyeq_{\rm disp})$ is a complete lattice.   
\end{corollary}
 \begin{proof}
 Let $A \subset \mathcal{M}_\infty/_\sim$ be a set which is bounded from above w.r.t.~$\preccurlyeq _{\rm disp}$, and denote by $A_0$ the set of respective representatives $\xi$ with $q_{\xi}(0^+)=0$. 
 Suppose that $\xi^\ast$ is an upper bound of $A_0$.  
 Then
 \[
{\rm TV}_{[0,1]}(A_0)\le {\rm TV}_{[0,1]}(q_{\xi^\ast})=q_{\xi^\ast}(1^-)-q_{\xi^\ast}(0^+)<\infty. 
 \]
Hence, $A_0$ is of bounded variation and has a supremum in $\mathcal{M}_\infty$ w.r.t.~$\preccurlyeq _{\rm disp}$ by Proposition~\ref{prop:existenceSup}. 
 Then passing to the quotient, $A$ has a unique supremum in $\mathcal{M}_\infty/_\sim$. 
 \end{proof}

It follows from Proposition~\ref{prop:existenceSup}
that every sequence $(\xi_n)_{n\in\N}\subset \mathcal{M}_\infty$ which is bounded from above w.r.t.~$\preccurlyeq_{\rm disp}$ has a supremum $\xi^\ast\in \mathcal{M}_\infty$ which is  determined 
by 
\[
q_{\xi^\ast}(u):={\rm TV}_{[0,u]}\left(\{q_{\xi_n}\colon n\in\N\}\right)\quad\mbox{ for all }u\in(0,1].
\]
Here, for $0\le u<v\le 1$, we denote by 
${\rm TV}_{[u,v]}\left(Q\right)
:=\sup_{\pi\in \Pi[u,v]}S_\pi(Q)$
the \emph{total variation} along $[u,v]$ of a non-empty set $Q$ of quantile functions, 
where $\Pi[u,v]$ is the set of all partitions of the form
$\pi=\{u=t_0<t_1<\cdots<t_n=v\}$ for some $n\in\N$, for which we set 
$S_\pi(Q):=\sum_{i=1}^n \sup_{q\in Q}(q(t_i)-q(t_{i-1}))$ with $q(0):=q(0^+)$ and $q(1):=q(1^-)$.  
It follows from Proposition~\ref{prop:dispOrderQuant} that  
\begin{equation}\label{eq:dispQuantiles}
\xi\preccurlyeq _{\rm disp} \eta\quad\mbox{if and only if}\quad 
    q_\xi(v)-q_\xi^+(u)\le q_\eta(v)-q_\eta^+(u)\mbox{ for all }0<u< v < 1,  
\end{equation}
where $q_\xi^+(u):=\inf\{x\in\mathbb{R}\colon \mathbb{P}(\xi\le x)>u\}$ is the right quantile function of $\xi$. 

In the following, we consider the family $\Phi_{\rm disp}:=\{\phi_{u,v}\colon 0<u< v < 1\}$ of functions $\phi_{u,v}\colon \mathcal{M}_\infty\to \mathbb{R}$ defined by $\phi_{u,v}(\xi):=q_\xi(v)-q_\xi^+(u)$.   
We identify $\Phi_{\rm disp}$ with the set $\Delta:=\{(u,v)\in \mathbb{R}^2\colon 0<u<v<1\}$ which we endow with the Euclidean topology.\footnote{Although $\Delta$ is not a complete metric space, it is completely metrizable as it is an open subset,  see~\cite[Theorem 24.12]{willard}. In particular, $\Delta$ is a Polish space with the Euclidean topology.} 
It follows from \eqref{eq:dispQuantiles} that the family $\Phi_{\rm disp}$ induces the dispersive order on $\mathcal{M}_\infty$.
By Proposition~\ref{lem:basicQuantile}, the mapping $(u,v)\mapsto \phi_{u,v}(\xi)$ is lower semicontinuous on $\Delta$ for all $\xi\in\mathcal{M}_\infty$. 
Hence, the ordered space $(\mathcal{M}_\infty,\preccurlyeq _{\rm disp})$ satisfies the assumptions of Section \ref{sec:mainResult}.
\begin{lemma}
 A function $\alpha\colon \Delta\to [0,\infty)$ is countably stable if there exists a non-decreasing function $\beta\colon [0,1]\to\mathbb{R}$ such that  
 $\alpha(u,v)=\beta(v)-\beta(u)$ for all $(u,v)\in \Delta$.    
\end{lemma}
\begin{proof}
  Suppose that $(\xi_n)_{n\in\N}\subset\mathcal{M}_\infty$ satisfies $\sup_{n\in\mathbb{N}}\{q_{\xi_n}(v)-q_{\xi_n}^+(u)\}\le \alpha(u,v)$ for all $(u,v)\in \Delta$. As an application of Proposition~\ref{prop:rightIncr} and Proposition \ref{prop:leftContTV}, we obtain that 
  \[
  {\rm TV}_{[0,1]}(\{q_{\xi_n}\colon n\in\N\})=\lim_{t\to 1^-} {\rm TV}_{[0,t]}(\{q_{\xi_n}\colon n\in\N\})\le \beta(1)-\beta(0)<\infty.
  \]
  By Proposition~\ref{prop:existenceSup}, the sequence $(\xi_n)_{n\in\mathbb{N}}$ has a supremum $\xi^\ast\in\mathcal{M}_\infty$ which is determined by
  \[
  q_{\xi^\ast}(u)={\rm TV}_{[0,u]}(\{q_{\xi_n}\colon n\in\N\})\quad\mbox{ for all }u\in (0,1].
  \]
  Then for $(u,v)\in \Delta$, we obtain from ${\rm TV}_{[0,v]}={\rm TV}_{[0,u]}+{\rm TV}_{[u,v]}$ that
  \begin{align*}
  q_{\xi^\ast}(v)-q_{\xi^\ast}^+(u)
  &\le {\rm TV}_{[0,v]}(\{q_{\xi_n}\colon n\in\N\})-{\rm TV}_{[0,u]}(\{q_{\xi_n}\colon n\in\N\})\\
  &={\rm TV}_{[u,v]}(\{q_{\xi_n}\colon n\in\N\})\le \beta(v)-\beta(u)=\alpha(u,v),    
  \end{align*}
  which shows that $\alpha$ is countable stable.
\end{proof} 

A set $Q$ of quantile functions is of bounded variation if ${\rm TV}_{[0,1]}(Q)<\infty$. 
\begin{lemma}\label{lem:dispOrdDirectedUp}
Let $A$ be a non-empty subset of $\mathcal{M}_\infty$.    
If $A$ is countably directed w.r.t.~$\preccurlyeq _{\rm disp}$, then the set $Q:=\{q_\xi\colon \xi\in A\}$ is of bounded variation and satisfies
\[
\sup_{\xi\in A}(q_\xi(v)-q_\xi^+(u))={\rm TV}_{[u,v]}(Q)\quad\mbox{for all }(u,v)\in \Delta.
\]
\end{lemma}
\begin{proof}
By contradiction, suppose that ${\rm TV}_{[0,1]}(Q)=\infty$. Then there exists a sequence $(\xi_n)_{n\in\N}\in A$ and partitions $\pi_1\subset \pi_2\subset \cdots$ of $[0,1]$ with $\lim_{n\to\infty} S_\pi(q_{\xi_n})=\infty$. 
 Since $A$ is countably directed, the sequence $(\xi_n)_{n\in\N}$ has a  supremum $\xi^\ast\in\mathcal{M}_\infty$. 
Then
\[
 \infty= \lim_{n\to\infty} S_{\pi_n}(q_{\xi_n})\le \sup_{n\to\infty} S_{\pi_n}(q_{\xi^\ast})\le {\rm TV}_{[0,1]}(q_{\xi^\ast}), 
 \]
in contradiction to $ {\rm TV}_{[0,1]}(q_{\xi^\ast})<\infty$ since $\xi^\ast$ is essentially bounded. This shows that $Q$ is of  bounded variation. 

Define $\alpha(u,v):=\sup_{\xi\in A}(q_\xi(v)-q_\xi^+(u))$.   
For $(u,v)\in \Delta$, it follows from Proposition~\ref{prop:rightIncr} that
\[
\alpha(u,v)\le {\rm TV}_{[u,v]}(Q).
\]
To prove the equality, suppose by contradiction that $\alpha(u,v)< {\rm TV}_{[u,v]}(Q)$. 
Then by Proposition~\ref{prop:rightIncr}, there exists a partition $\pi:=\{u=t_0<t_1<\cdots<t_n=v\}$ with
\[
\alpha(u,v) < \sum_{i=1}^n \sup_{\xi\in A}(q_\xi(t_{i})-q_\xi^+(t_{i-1}))<\sum_{i=1}^n (q_{\xi_i}(t_{i})-q_{\xi_i}^+(t_{i-1}))
\]
for some $\xi_1,\xi_2,\cdots,\xi_n\in A$. 
Since $A$ is countably directed, the exists $\xi^\ast\in A$ such that $\xi_1,\dots,\xi_n \preccurlyeq _{\rm disp}\xi^\ast$. 
Then
\[
\alpha(u,v) < \sum_{i=1}^n (q_{\xi^\ast}(t_{i})-q_{\xi^\ast}^+(t_{i-1}))\le q_{\xi^\ast}(v)-q_{\xi^\ast}^+(u)\le \alpha(u,v),
\]
which is a contradiction. 
\end{proof}

For a function $\psi\colon \mathcal{M}_\infty\to (-\infty,\infty]$ 
with sublevel sets $A_s:=\{\xi\in \mathcal{M}_\infty\colon \psi(\xi)\le s\}$, we define $\alpha_{\min}(s,u,v):=\sup_{\xi\in A_s}(q_\xi(v)-q_\xi^+(u))$ and $\alpha^+_{\min}(s,u,v):=\inf_{s<t}\alpha_{\min}(t,u,v)$ for all $s\in\R$ and $(u,v)\in\Delta$. 

Let $I:=\{s\in\mathbb{R}\colon A_s\neq\emptyset\}$. If $\psi$ has countably upwards directed sublevel sets, it follows from Lemma~\ref{lem:dispOrdDirectedUp} that for every $s\in I$,  the set $Q_s:=\{q_\xi\colon \xi\in A_s\}$ is of bounded variation, and $\alpha_{\min}(s,u,v)=\beta(s,v)-\beta(s,u)$ for all $(u,v)\in \Delta$ where $\beta(s,u):={\rm TV}_{[0,u]}(Q_s)$.  
Moreover, the mapping $\beta\colon\R\times (0,1) \to\R$ is non-decreasing in the second argument and satisfies 
$\beta(s,v)-\beta(s,u)\le \beta(s^\prime,v)-\beta(s^\prime,u)$
for all $s\le s^\prime$. As a direct application of Theorem~\ref{thm:main}, we obtain the following representation result for maxitive functions w.r.t.~the dispersive stochastic order. 

\begin{corollary}\label{cor:maxDisp}
Let $\psi\colon \mathcal{M}_\infty\to(-\infty,\infty]$ be a function with sublevel sets $(A_s)_{s\in\mathbb{R}}$.  
Then, the following conditions are equivalent:
\begin{enumerate}
    \item $\psi$ is monotone and $(A_s)_{s\in\mathbb{R}}$ are countably upwards directed   w.r.t.~$\preccurlyeq _{\rm disp}$.
    \item There exists a function $\beta\colon I\times (0,1)\to \mathbb{R}$ which is non-decreasing in the second argument, and satisfies $\beta(s,v)-\beta(s,u)\le \beta(s^\prime,v)-\beta(s^\prime,u)$ for all $s\le s^\prime$ and $(u,v)\in \Delta$. Moreover, it holds 
    \[
    \psi(\xi)=\inf\left\{ s\in \mathbb{R} \colon \sup_{(u,v)\in \Delta}\{q_\xi(v)-q_\xi(u^+)-\alpha(s,u,v)\}\le 0\right\}\quad\mbox{ for all }\xi\in\mathcal{M}_\infty,
    \]
    where $\alpha(s,u,v):=\beta(s,v)-\beta(s,u)$ for all $s\in I$, and $\alpha(s,u,v):=\infty$ otherwise. 
    The infimum is attained for all~$\xi\in\mathcal{M}_\infty$ with $\psi(\xi)<\infty$. 
\item For every $s\in I$, the set $Q_s:=\{q_\xi\colon \xi\in A_s\}$ is of bounded variation and $\alpha_{\min}(s,u,v)={\rm TV}_{[u,v]}(Q_s)$ for all $(u,v)\in \Delta$. Moreover, it holds
    \[
\psi(\xi)=\sup_{(u,v)\in \Delta}G(q_\xi(v)-q_\xi^+(u),u,v)\quad\mbox{ for all }\xi\in \mathcal{M}_\infty,
\]
where  $G(t,u,v):=\inf\{s\in\mathbb{R}\colon t\le \alpha_{\min}^+(s,u,v)\}$ for all $(t,u,v)\in \mathbb{R}\times \Delta$.
\end{enumerate}
\end{corollary}

We conclude with an example of a maxitive function for the dispersive order.

\begin{example}
Consider the function \(\beta\colon [0,\infty) \times (0,1) \to \mathbb{R}\), defined by \(\beta(s,u) := s u\).  The function  
\[
\psi(\xi) := \sup_{(u,v) \in \Delta} \frac{q_{\xi}(v) - q_{\xi}(u^+)}{v - u} \in [0,\infty]
\]  
satisfies the equivalent conditions of Corollary~\ref{cor:maxDisp} and is therefore bounded countably maxitive on $\mathcal{M}_\infty/_\sim$.
The value \(\psi(\xi)\) corresponds to the Lipschitz constant of \(q_\xi\). In particular, if \(q_\xi\) is differentiable, it holds that  
\[
\psi(\xi) = \sup_{s \in (0,1)} q_\xi^\prime(s).
\]  
In reliability analysis, the derivative \(q_\xi^\prime(s)\) is referred to as the \emph{quantile density function} and measures the rate of change of the quantile function; see \cite{unnikrishnan}.  When \(\xi\) models the uncertain time to failure of a system, \(q_\xi(u)\) corresponds to the time at which the survival probability of the system is \(1 - u\). A high value of \(q_\xi^\prime(u)\) implies that small reductions in survival probability correspond to larger increases in time, indicating a slower rate of reliability decline over time. Consequently, \(\psi(\xi)\) acts as a measure of the speed at which the reliability of the system decreases, capturing the steepest rate of decline.  
\end{example}

\section{Conclusion}
Maxitive functions are widely used in decision-making under risk and uncertainty, particularly for worst-case or best-case evaluations. In this paper, we extend the concept of maxitive functions to a broad class of preorders, providing a unified framework that goes beyond the classical pointwise order. Specifically, our results address the main stochastic orders in the literature, including the usual stochastic order, the increasing convex order, and the dispersive order. 

By transitioning from the pointwise order to weaker stochastic orders, we develop a framework that covers a broader class of maxitive functions. Our main result is a representation theorem for bounded countably maxitive functions. In the translation-invariant case, we obtain representation results in terms of penalized worst-case evaluations. Notably, this generalization also includes widely used decision-making tools such as the value at risk.

\begin{appendix}
\section{Auxiliary results for maxitive risk measures on $C_b(S)$}
Let $(S,d)$ be a Polish space, and 
denote by $\mathcal{X}$ a convex cone of functions $f\colon S\to \mathbb{R}$ with $C_b(S)\subset \mathcal{X}$. Here,
$C_b(S)$ denotes the space of all bounded continuous functions. 
The set $\mathcal{X}$ is endowed with the pointwise order.
For a function $\psi\colon\mathcal{X}\to(-\infty,\infty]$, 
we frequently deal with the following properties:
\begin{itemize}
\item[(T)] $\psi(f+c)=\psi(f)+c$ for all $c\in\R$.
\item[(M)] $\psi(f)\le \psi(g)$ whenever $f\le g$.
\item[(C)] $\psi(\lambda f + (1-\lambda)g)\le \lambda\psi( f ) + (1-\lambda)\psi(g)$ for all $\lambda\in[0,1]$.
\end{itemize}
The function $\psi$ is \emph{continuous from below}  
if $\psi(f)=\lim_{n\to\infty}\psi(f_n)$ for all sequences $(f_n)_{n\in\mathbb N}\subset\mathcal{X}$ with $f_n\le f_{n+1}$ for all $n\in\mathbb{N}$ and $f=\sup_{n\in\mathbb{N}} f_n\in\mathcal{X}$.

For $\psi\colon \mathcal{X}\to(-\infty,\infty]$, we consider the \emph{acceptance set} $\mathcal{A}:=\{f\in\mathcal{X}\colon \psi(f)\le 0\}$, and 
define the \emph{convex conjugate} $\psi_{C_b(S)}^\ast\colon {\rm ca}_1^+(S)\to[0,\infty]$ by
\[
\psi^\ast_{C_b(S)}(\mu):=\sup_{f\in C_b(S)}\left\{\int f\, {\rm d}\mu - \psi(f) \right\},
\]
where ${\rm ca}_1^+(S)$ denotes the set of all Borel probability measures on $S$. 
In the particular, for the Dirac measure $\delta_x$, we get
\[
I_{\min}(x):=\psi^\ast_{C_b(S)}(\delta_x)=
\sup_{f\in C_b(S)}\{f(x)-\psi(f)\}\quad\mbox{for all }x\in S.
\]
The following result was recently shown by Delbaen \cite{delbaen,delbaen2}; see also \cite{max} for an alternative proof. 

\begin{theorem}\label{lem:Delbaen}
Let $\psi\colon C_b(S)\to\mathbb{R}$ be a function satisfying (T), (M) and (C). 
Then, the following statements are
equivalent:
\begin{itemize}
\item[(i)] $\psi$ is continuous from below.
\item[(ii)] $\psi$ admits the representation
\[
\psi(f)=\sup_{\mu\in {\rm ca}_1^+(S)}\left\{\int f\, {\rm d}\mu - \psi_{C_b(S)}^\ast(\mu)\right\}.
\]
\end{itemize}
\end{theorem}
In combination with \cite[Theorem 2.1]{kupper}, we obtain the following result. Whereas the representation result in \cite{kupper} relies on continuity from above, here we rely on the weaker condition of continuity from below, which is guaranteed by countable maxitivity.
\begin{theorem}\label{thm:maxCb}
Let $\psi\colon C_b(S)\to\mathbb{R}$ be a function satisfying (T) and (M). 
Then, the following statements are
equivalent:
\begin{itemize}
\item[(i)]  $\psi$ is countably maxitive, i.e., $\psi(f)=\sup_{n\in\N}\psi(f_n)$ for every  $(f_n)_{n\in\mathbb{N}}\subset  C_b(S)$ with $f=\sup_{n\in\N}f_n\in  C_b(S)$.
\item[(ii)] $\sup_{n\in\N} f_n\in \mathcal{A}$ for every  $(f_n)_{n\in\mathbb{N}}\subset \mathcal{A}$ with $\sup_{n\in\N} f_n\in C_b(S)$. 
\item[(iii)] $\psi$ admits the representation
\[
\psi(f)=\sup_{x\in S}\left\{f(x) - I_{\min}(x)\right\}\quad\mbox{for all }f\in C_b(S).
\]

\item[(iv)] $\psi$ satisfies (C) and is continuous from below. Moreover, 
\[
\psi_{C_b(S)}^\ast(\mu)
=\int I_{\min}(x)\,\mu({\rm d}x) \quad\mbox{for all }\mu\in {\rm ca}_1^+(S).
\]
\item[(v)] There exists a function $\gamma\colon S\to[0,\infty]$ such that 
\[
\mathcal{A}=\{f\in C_b(S)\colon f\le \gamma\}.
\]
\end{itemize}
\end{theorem}
\begin{proof}
The argumentation is similar to that in \cite[Theorem 2.1]{kupper}. Similar to the proof of \cite[Theorem 2.1]{kupper}, it holds
\begin{equation}\label{eq:esssup}
\psi^\ast(\delta_{\cdot})
=\esssup_\mu\mathcal{A}\quad\mbox{$\mu$-almost surely}
\end{equation}
for all $\mu\in {\rm ca}_1^+(S)$, where $\esssup_\mu\mathcal{A}$ denotes the essential supremum of $\mathcal{A}$ with respect to the probability measure $\mu$.

$(i) \Rightarrow(ii)$:  For \( (f_n)_{n \in \mathbb{N}} \subset \mathcal{A} \) with \( \sup_{n \in \mathbb{N}} f_n \in C_b(S) \), we have 
\[
\psi\Big(\sup_{n \in \mathbb{N}} f_n\Big) = \sup_{n \in \mathbb{N}} \psi(f_n) \leq 0,
\]
and therefore \( \sup_{n \in \mathbb{N}} f_n \in \mathcal{A} \).

$(ii) \Rightarrow(i)$: Let \( (f_n)_{n \in \mathbb{N}} \subset C_b(S) \) with \( f := \sup_{n \in \mathbb{N}} f_n \in C_b(S) \). 
Since by (T), \( f_n - \sup_{m \in \mathbb{N}} \psi(f_m) \in \mathcal{A} \) for all \( n \in \mathbb{N} \), we obtain that \( f - \sup_{m \in \mathbb{N}} \psi(f_m) \in \mathcal{A} \). By (T) and (M), it holds
\[
\psi(f) \leq \sup_{n \in \mathbb{N}} \psi(f_n) \leq \psi(f).
\]

$(i)\Rightarrow(iv):$ Condition (i) ensures that \( \psi \) is (finitely) maxitive and continuous from below. In particular, since \( \psi \) is maxitive, it satisfies (C) due to \cite[Proposition 2.1]{kupper}.  
Fix \( \mu \in {\rm ca}_1^+(S) \).  
Since \( \mathcal{A} \) is directed upwards, there exists a sequence \( (f_n)_{n \in \mathbb{N}} \subset \mathcal{A} \) with \( f_n \leq f_{n+1} \) for all \( n \in \mathbb{N} \) such that \( f_n \nearrow \esssup_\mu \mathcal{A} \) \( \mu \)-almost surely.  
By \eqref{eq:esssup}, we have that \( \esssup_\mu \mathcal{A} = I_{\min} \) \( \mu \)-almost surely. Hence, by monotone convergence,  
\[
\psi_{C_b(S)}^\ast(\mu) = \sup_{f \in \mathcal{A}} \int f \, {\rm d}\mu \geq \sup_{n \in \mathbb{N}} \int f_n \, {\rm d}\mu =
\int I_{\min}(x) \, \mu({\rm d}x).
\]
Moreover,
\[
\psi_{C_b(S)}^\ast(\mu) = \sup_{f \in \mathcal{A}} \int f \, {\rm d}\mu \leq \int \esssup_\mu \mathcal{A} \, {\rm d}\mu = \int I_{\min}(x) \, \mu({\rm d}x).
\]

$(iv)\Rightarrow(iii):$ For every $f\in C_b(S)$, it follows from   Theorem~\ref{lem:Delbaen} that 
\begin{align*}
\psi(f)&=\sup_{\mu\in {\rm ca}_1^+(S)}\left\{\int f\, {\rm d}\mu - \psi_{C_b(S)}^\ast(\mu)\right\}=\sup_{\mu\in {\rm ca}_1^+(S)}\left\{\int f\, {\rm d}\mu - \int I_{\min}(x)\,\mu({\rm d}x)\right\}\\
&=\sup_{\mu\in {\rm ca}_1^+(S)}\left\{\int f(x) - I_{\min}(x)\,\mu({\rm d}x)\right\}=\sup_{x\in S}\left\{f(x) - I_{\min}(x)\right\}.
\end{align*}

$(iii)\Rightarrow(v):$ From the representation, $f\in\mathcal{A}$ if and only if $f(x)\le I_{\min}(x)$ for all $x\in S$.
Hence, $\mathcal{A}=\{f\in C_b(S)\colon f\le \gamma\}$ with 
$\gamma=I_{\min}$. 

$(v)\Rightarrow(ii):$ Let $(f_n)_{n\in\mathbb{N}}\subset\mathcal{A}$ with $\sup_{n\in\N} f_n\in C_b(S)$. Since $f_n\le\gamma$ for all $n\in\N$,
it follows that $\sup_{n\in\N} f_n\le \gamma$. 
This shows that $f\in\mathcal{A}$. 
\end{proof}

We denote by $L(S)$ the set of all lower semicontinuous functions $f\colon S\to \mathbb{R}$.
\begin{corollary}\label{cor:repMax2}
Suppose that $\psi\colon {L}(S)\to (-\infty,\infty]$ satisfies $\psi(0)\in\mathbb{R}$, properties (T) and (M), and that $\psi(f)=\sup_{n\in\mathbb{N}}\psi(f_n)$ for every sequence $(f_n)_{n\in\mathbb{N}}\subset L(S)$ such that $f=\sup_{n\in\mathbb{N}}f_n\in L(S)$. Then, 
\[
\psi(f)=\sup_{x\in S}\{f(x)-I_{\min}(x)\} \quad \text{for all } f\in {L}(S).
\]
\end{corollary}

\begin{proof}
Since $\psi(0)\in\mathbb{R}$, property (T) implies that the restriction of $\psi$ to $C_b(S)$ is real-valued. Thus, by Theorem~\ref{thm:maxCb}, we have
\[
\psi(f)=\sup_{x\in S}\{f(x)-I_{\min}(x)\} \quad \text{for all } f\in C_b(S).
\]

Now, let \( f \in L(S) \) be bounded from below. In this case, there exists a sequence \((f_n)_{n \in \mathbb{N}} \subset C_b(S)\) such that \( f_n \nearrow f \). For instance, the sequence defined by
\[
f_n(x) := \inf\left\{ f(y) + n d(x,y) : y \in S \right\} \wedge n
\]
satisfies this property. Since \(\psi\) is continuous from below, we have
\begin{align*}
\psi(f) &= \sup_{n \in \mathbb{N}} \psi(f_n) = \sup_{n \in \mathbb{N}} \sup_{x \in S} \{f_n(x) - I_{\min}(x)\} \\
&= \sup_{x \in S} \sup_{n \in \mathbb{N}} \{f_n(x) - I_{\min}(x)\} = \sup_{x \in S} \{f(x) - I_{\min}(x)\}.
\end{align*}

Finally, consider a general \( f \in L(S) \). Since \(\psi\) is maxitive, for every \( n \in \mathbb{N} \), we have
\begin{align*}
\psi(f) \vee \psi(-n) &= \psi(f \vee (-n)) \\
&= \sup_{x \in S} \{ (f(x) \vee (-n)) - I_{\min}(x) \} \\
&= \sup_{x \in S} \{ (f(x) - I_{\min}(x)) \vee (-n - I_{\min}(x)) \}.
\end{align*}
Letting \( n \to \infty \), we obtain that
\(\psi(f) = \sup_{x \in S} \{f(x) - I_{\min}(x)\}\).
\end{proof}

\section{Auxiliary results for  quantile functions}\label{sec:appendixC}

Suppose that $(\Omega,\mathcal{F},\mathbb{P})$ is an atomless probability space. Recall that $\mathcal{M}$ denotes the set of all random variables $\xi\colon \Omega\to\R$, where two of them are identified if they have the same law. Given $\xi\in\mathcal{M}$, we define the cumulative distribution function by $F_\xi(x):=\mathbb{P}(\xi\le x)$ and set $F_\xi(x^-)=\sup_{y<x} F_\xi(y)$ for all $x\in\R$. Moreover, the quantile function $q_\xi\colon (0,1)\to\R$ 
and upper quantile function $q_\xi^+\colon (0,1)\to\R$
are defined by
\[
q_\xi(t):=\inf\{x\in\mathbb{R}\colon t\le F_\xi(x)\}
\quad\mbox{and}\quad
q_\xi^+(t):=\inf\{x\in\mathbb{R}\colon t< F_\xi(x)\}.
\]
\begin{lemma}\label{lem:basicQuantile}
For all $\xi\in\mathcal{M}$ and $t\in(0,1)$, the following statements hold:
\begin{itemize}
\item[(i)] $[q_\xi(t),\infty)=\{x\in\mathbb{R}\colon t\le F_\xi(x)\}$.
\item[(ii)] $(-\infty,q^+_\xi(t)]=\{x\in \mathbb{R}\colon t\ge F_\xi(x^-)\}$.
\item[(iii)] $\sup_{s<t}q_\xi(s)=\sup_{s<t}q^+_\xi(s)=q_\xi(t)$.
\item[(iv)] $\inf_{s>t}q_\xi(s)=\inf_{s>t}q^+_\xi(s)=q^+_\xi(t)$.
\item[(v)] The function $q_\xi\colon(0,1)\to\R$ is left-continuous.
\item[(vi)] The function $q_\xi^+\colon(0,1)\to\R$ is right-continuous.
\item[(vii)] The function $(s,r)\mapsto q_\xi(r)-q_\xi^+(s)$ is lower semicontinuous on $(0,1)^2$.
\end{itemize}
\end{lemma}
\begin{proof}
(i) and (ii) follow from the definitions of $q$ and $q^+$, using that $F_\xi$ is non-decreasing and right-continuous.

For $a:=\sup_{s<t}q_\xi(s)$, it follows from (i) that
 \[
     [a,\infty)=\underset{s<t}\bigcap [q_\xi(s),\infty)
     =\underset{s<t}\bigcap F_\xi^{-1}([s,\infty)])
     = F_\xi^{-1}([t,\infty)])
     =[q_\xi(t),\infty),
 \]
 where $F_\xi^{-1}([s,\infty))=\{x\in\mathbb{R}\colon s\le F_\xi(x)\}$. 
 This shows $\sup_{s<t}q_\xi(s)=q_\xi(t)$.  
 For $s<t$, it holds $q_\xi(s)\le q_\xi^+(s)\le q_\xi(t)$, so that 
 $
 q_\xi(t)=\sup_{s<t}q_\xi(s)\le \sup_{s<t}q_\xi^+(s) \le q_\xi(t)$, which 
shows (iii). 

 Define $G(s):=F_\xi(s^-)$ for all $s\in\mathbb{R}$. For $b:=\inf_{s>t}q_\xi^+(s)$, it follows from (ii) that
 \[
 (-\infty,b]=\underset{s>t}\bigcap (-\infty,q_\xi^+(s)]
     =\underset{s>t}\bigcap G^{-1}((-\infty,s])
     = G^{-1}((-\infty,t])
     =(-\infty,q_\xi^+(t)],
 \]
 where $G^{-1}([s,\infty))=\{x\in\mathbb{R}\colon s\le G(x)\}$.
 This shows $\inf_{s>t}q_\xi^+(s)=q_\xi^+(t)$. 
 For $s>t$, it holds $q_\xi^+(s)\ge q_\xi(s)\ge q_\xi^+(t)$, so that 
 $q_\xi^+(t)= \inf_{s>t}q_\xi^+(s)\ge \inf_{s>t}q_\xi(s)\ge q_\xi^+(t)$, which shows (iv).
 
 Finally, (v), (vi) and (vii) are consequences of (iii) and (iv). 
\end{proof}

\begin{lemma}\label{prop:dispOrderQuant}
For $\xi,\eta\in\mathcal{M}$, the following statements are equivalent:
\begin{itemize}
    \item[(i)] $q_\xi(u)-q_\xi(t)\le q_\eta(u)-q_\eta(t)$ for all $0<t\le u<1$.
    \item[(ii)] $q_\xi(u)-q_\xi^+(t)\le q_\eta(u)-q_\eta^+(t)$ for all $0<t< u<1$.
\end{itemize}
\end{lemma}
\begin{proof}
  $(ii)\Rightarrow(i)$:  Let $0<t\le u<1$. For $s\in (0,t)$, it holds
  $q_\xi(u)-q_\xi^+(s)\le q_\eta(u)-q_\eta^+(s)$. Letting $s\to t^-$, 
  it follows from (iii) in Lemma~\ref{lem:basicQuantile} that 
  $$q_\xi(u)-q_\xi(t)\le q_\eta(u)-q_\eta(t).$$
  
  $(i)\Rightarrow(ii)$: Let $0<t< u<1$. For $s\in (t,u)$, it holds
  $q_\xi(u)-q_\xi(s)\le q_\eta(u)-q_\eta(s)$. 
  Letting $s\to t^+$, it follows from (iv) in Lemma~\ref{lem:basicQuantile} that
  $q_\xi(u)-q_\xi^+(t)\le q_\eta(u)-q_\eta^+(t)$.
\end{proof}

\begin{lemma}\label{lem:basicQuantile2}
If $q\colon (0,1)\to \mathbb{R}$ is non-decreasing and left-continuous, then there exists $\xi\in\mathcal{M}$  such that $q(u)=q_\xi(u)$ for all $u\in(0,1)$.  
\end{lemma}
\begin{proof}
 If $(\Omega,\mathcal{F},\mathbb{P})$ is atomless, then there exists a uniformly distributed random variable $U$ with values in $(0,1)$ defined on  $(\Omega,\mathcal{F},\mathbb{P})$, see e.g.~\cite[Theorem 1]{delbaen0} and \cite[Proposition A.27]{follmer}.  
 Then $\xi=q(U)$ is as desired. To that end, define $F\colon \mathbb{R}\to [0,1]$ by
 \[
 F(x):=0\vee\sup\{s\in(0,1)\colon x\ge q(s)\},
 \]
 which satisfies
 \begin{itemize}
     \item[(1)] $q(F(x))\le x$ for all $x\in \mathbb{R}$ (where we set $q(0):=q(0^+)$ and $q(1):=q(1^-)$),
     \item[(2)] $F(q(s))\ge s$ for all $s\in (0,1)$.
 \end{itemize}   
 For $x\in\mathbb{R}$, it follows from (2) that $\{q(U)\le x \}\subset \{F(q(U))\le F(x)\}\subset \{U\le F(x)\}$. Hence, for every $u\in(0,1)$,
 \begin{align*}
  q_\xi(u)&=\inf\{x\in\mathbb{R}\colon \mathbb{P}(q(U)\le x)\ge u\}\\
  &\ge \inf\{x\in\mathbb{R}\colon \mathbb{P}(U\le F(x))\ge u\}\\
  &=\inf\{x\in\mathbb{R}\colon F(x)\ge u\}.
 \end{align*}
 Let $a:=\inf\{x\in\mathbb{R}\colon F(x)\ge u\}$, and $x_n\downarrow a$ with $F(x_n)\ge u$. 
 It follows from (1) that $x_n\ge q(u)$, and therefore $q_\xi(u)\ge a\ge q(u)$
 by letting $n\to\infty$. 

 As for the other inequality, we have that $\mathbb{P}(q(U)\le q(u))\ge \mathbb{P}(U\le u)=u$. Hence, it follows from the definition of $q_\xi(u)$ that $q_\xi(u)\le q(u)$.
\end{proof}

Recall that for $\xi\in\mathcal{M}_1$, the integrated quantile function $Q_\xi\colon (0,1)\to\mathbb{R}$ is defined by
\[
Q_\xi(u):=\int_u^1 q_\xi(v)\,{\rm d}v.
\]
\begin{lemma}\label{lem:basicQuantile3}
    If $Q\colon (0,1)\to \mathbb{R}$ is concave and satisfies $Q(0^+)\in\mathbb{R}$ and $Q(1^-)=0$, then there exists $\xi\in\mathcal{M}_1$ such that $Q(u)=Q_\xi(u)$ for all $u\in(0,1)$.
\end{lemma}
\begin{proof}
    Since $Q$ is concave,  there exist  a non-decreasing function $q\colon (0,1)\to\mathbb{R}$ and $c\in(0,1)$ such that
    \[
    Q(u)=Q(c)-\int_c^u q(v){\rm d}v,
    \]
    see e.g.~\cite[Theorem 12A]{roberts}. By replacing $q$ by its left-continuous modification, we can assume that $q$ is left-continuous.  
    Using that $Q(1^-)=0$, we get that 
    \[
    Q(u)=\int_u^1 q(v)\,{\rm d}v.
    \]
   Since $q$ is left-continuous, it follows from Lemma~\ref{lem:basicQuantile2} that $q=q_\xi$ for some $\xi\in\mathcal{M}$. 
   It remains to show that $\xi\in\mathcal{M}_1$. 
   We have
   \begin{align*}
      \mathbb{E}_{\mathbb{P}}[|\xi|]&=
   \int_0^1 |q_\xi(v)|{\rm d}v\\
   &=-\int_0^{F_\xi(0)} q_\xi(v){\rm d}v+
   \int_{F_\xi(0)}^1 q_\xi(v){\rm d}v\\
   &=-Q(0^+)+ 2 Q(F_\xi(0)) <\infty.    
   \end{align*}
   Hence, $\xi$ has finite first moment, and consequently $\xi\in\mathcal{M}_1$. 
\end{proof}

\section{Auxiliary results for the total variation}\label{sec:appendixB}
 We consider the setting of Subsection~\ref{sec:disp order}.

\begin{proposition}\label{prop:leftContTV}
The mapping  $t\mapsto {\rm TV}_{[0,t]}(Q)$  is non-decreasing and  left-continuous on $(0,1]$.     
\end{proposition}
\begin{proof}
By definition, the mapping $t\mapsto {\rm TV}_{[0,t]}(Q)$ is non-decreasing. 
    Given $t\in (0,1]$, we fix a partition $\pi=\{t_0<t_1<\cdots<t_n\}$ of $[0,t]$.  
    For every $u\in (0,t)$, we consider the partition $\pi_u:=(\pi\cap [0,u])\cup\{u\}\in\Pi[0,u]$. 
    Then since every $q\in Q$ is non-decreasing and left-continuous, we have that
    \begin{align*}
    \sup_{u< t}S_{\pi_u}(Q)&=\sup_{u\in (t_{n-1},t)}S_{\pi_u}(Q)\\
    &=\sum_{i=1}^{n-1} \sup_{q\in Q}(q(t_i)-q(t_{i-1})) + \sup_{u\in (t_{n-1},t)}\sup_{q\in Q}(q(u)-q(t_{n-1}))\\
    &=\sum_{i=1}^{n-1} \sup_{q\in Q}(q(t_i)-q(t_{i-1})) + \sup_{q\in Q}\sup_{u\in (t_{n-1},t)}(q(u)-q(t_{n-1}))\\
    &=\sum_{i=1}^{n-1} \sup_{q\in Q}(q(t_i)-q(t_{i-1})) + \sup_{q\in Q}(q(t)-q(t_{n-1}))\\
    &=S_{\pi}( Q).
    \end{align*}
    Since ${\rm TV}_{[0,u]}(Q)=\sup_{\pi\in\Pi[0,u]}S_{\pi}(Q)=\sup_{\pi\in\Pi[0,t]}S_{\pi_u}(Q)$ for all $u<t$,
    it follows that 
    \begin{align*}
    \lim_{u\uparrow t}{\rm TV}_{[0,u]}(Q)&=\sup_{u<t}\sup_{\pi\in\Pi[0,t]}S_{\pi_u}(Q)=\sup_{\pi\in\Pi[0,t]}\sup_{u<t}S_{\pi_u}(Q)\\
    &=\sup_{\pi\in\Pi[0,t]}S_{\pi}(Q)= {\rm TV}_{[0,t]}(Q).\qedhere
    \end{align*}
\end{proof}

For $[u,v]\subset [0,1]$ and $\pi\in\Pi[u,v]$, we define $\underline{S}_{\pi}(Q):=\sum_{i=1}^n \sup_{q\in Q}(q_\xi(t_i)-q_\xi^+(t_{i-1}))$.

\begin{proposition}\label{prop:rightIncr} For every $[u,v]\subset[0,1]$, it holds
$${\rm TV}_{[u,v]}(Q)=\sup_{\pi\in\Pi[u,v]} \underline{S}_{\pi}(Q).$$
\end{proposition}
 \begin{proof}
By contradiction, we assume that $\sup_{\pi\in\Pi[u,v]} S_{\pi}(Q)>\sup_{\pi\in\prod[u,v]} \underline{S}_{\pi}(Q)$. Then, there exists $\bar\pi=\{u=t_0<t_1<\cdots<t_n=v\}\in\Pi[u,v]$ with
\[
S_{\bar\pi}(Q)>\sup_{\pi\in\Pi[0,t]} \underline{S}_{\pi}(Q).
\]
Fix $\varepsilon>0$ such that $S_{\bar\pi}(Q)-\varepsilon>\sup_{\pi\in\Pi[0,t]} \underline{S}_{\pi}(Q)$.
For every $i\in\{1,2,\cdots,n\}$, the mapping $t\mapsto \sup_{q\in Q}(q(t)-q(t_{i-1}))$ is left-continuous.  
Hence, there exists $u_i\in [t_{i-1},t_i)$ such that
\[
\sup_{q\in Q}(q(u_i)-q(t_{i-1}))>\sup_{q\in Q}(q(t_i)-q(t_{i-1}))-\frac{\varepsilon}{n}.
\]
Since $q^+(u_i)\le q(t_i)$ for all $i$, we get
\[
\sup_{q\in Q}(q(u_i)-q^+(u_{i-1}))\ge\sup_{q\in Q}(q(t_i)-q(t_{i-1}))-\frac{\varepsilon}{n}\quad\mbox{for all }i\in\{1,2,\cdots,n\},
\]
where $u_0:=t_0=u$. 
For the partition $\pi^\prime:=\{u=u_0<\cdots<u_n<t_n=v\}\in\Pi[u,v]$, we obtain  
\[
 \underline{S}_{\pi^\prime}(Q)\ge  S_{\bar\pi}(Q)-\varepsilon
 >\sup_{\pi\in\Pi[0,t]} \underline{S}_{\pi}(Q),
\]
which is a contradiction.
 \end{proof}

\end{appendix}

\end{document}